\documentclass[11pt,reqno]{amsart}

\usepackage[dvipsnames]{xcolor}
\usepackage{tikz}
\usepackage{stmaryrd}
\usetikzlibrary{matrix,arrows,decorations.pathmorphing}
\usepackage{graphicx}
\usepackage{mathpazo}
\usepackage{setspace}
\usepackage{euler}
\usepackage{amssymb}
\usepackage[stable]{footmisc}
\usepackage{amsmath,mathtools}
\usepackage{fullpage}
\usepackage{caption}
\usepackage{amsthm}
\usepackage{mathrsfs}
\usepackage{thmtools}
\usepackage{blkarray}
\usepackage{multirow}
\usepackage{picinpar} 
\usepackage{tikz-cd}
\usepackage{color}
\usepackage{verbatim}
\usepackage{hyperref}
\hypersetup{colorlinks=true, linkcolor=black, citecolor = OliveGreen, urlcolor = black}
\usepackage{amssymb}
\usepackage{amsmath}
\usepackage{enumitem,colonequals}
\usepackage{color}
\usepackage{mathrsfs}
\usepackage[all]{xy}
\usepackage{comment}
\usepackage{mathpazo}
\usepackage{euler}
\usepackage[margin=1.in]{geometry}

\newcommand{\mR}{\mathbb{R}}
\newcommand{\mQ}{\mathbb{Q}}

\newcommand{\mC}{\mathbb{C}}
\newcommand{\mG}{\mathbb{G}}
\newcommand{\mN}{\mathbb{N}}

\newcommand{\mP}{\mathbb{P}}

\newcommand{\mU}{\mathbb{U}}

\newcommand{\cC}{\mathcal{C}}

\newcommand{\cT}{\mathcal{T}}

\newcommand{\cO}{\mathcal{O}}

\newcommand{\wt}[1]{\widetilde{#1}}

\newcommand{\cL}{\mathcal{L}}

\usepackage[OT2,T1]{fontenc}
\DeclareSymbolFont{cyrletters}{OT2}{wncyr}{m}{n}
\DeclareMathSymbol{\Sha}{\mathalpha}{cyrletters}{"58}

\DeclareMathSymbol{\Sha}{\mathalpha}{cyrletters}{"58}

\newcommand{\brk}[1]{ \left\lbrace #1 \right\rbrace }
\newcommand{\pwr}[1]{ \left( #1 \right) }

\newcommand{\sH}{{\mathscr H}}

\newcommand{\sO}{{\mathscr O}}

\newcommand{\sU}{{\mathscr V}}
\newcommand{\sV}{{\mathscr V}}
\newcommand{\uusV}{\underline{\underline{\mathscr V}}}
\newcommand{\sW}{{\mathscr W}}
\newcommand{\sX}{{\mathscr X}}
\newcommand{\sY}{{\mathscr Y}}
\newcommand{\sZ}{{\mathscr Z}}

\RequirePackage{mathrsfs} 

\theoremstyle{theorem}
\numberwithin{equation}{subsection}
\newtheorem{thmx}{\textsc{Theorem}}

\newtheorem{corox}[thmx]{\textsc{Corollary}}
\newtheorem{theorem}[subsection]{Theorem}

\newtheorem{lemma}[subsection]{Lemma}
\newtheorem{corollary}[subsection]{Corollary}
\newtheorem{conjecture}[subsection]{Conjecture}
\newtheorem{prop}[subsection]{Proposition}

\numberwithin{equation}{subsection}

\theoremstyle{definition}
\newtheorem{definition}[subsection]{\text{Definition}}
\newtheorem{example}[subsection]{Example}
\newtheorem{remark}[subsection]{Remark}

\theoremstyle{remark}

\numberwithin{equation}{subsection} \numberwithin{figure}{section}

\DeclareMathOperator{\an}{an}

 \DeclareMathOperator{\Spec}{Spec}

\DeclareMathOperator{\Hilb}{Hilb}

\DeclareMathOperator{\Hom}{Hom}

\DeclareMathOperator{\Int}{Int}

\DeclareMathOperator{\triv}{triv}
\DeclareMathOperator{\Sp}{Sp}

\DeclareMathOperator{\red}{red}

\newcommand{\cdef}[1]{\textsf{\textit{#1}}}

\renewcommand{\leq}{\leqslant}

\renewcommand{\geq}{\geqslant}

\DeclareMathOperator{\Mor}{Mor}
\DeclareMathOperator{\REMor}{Mor_{ext}}
\DeclareMathOperator{\GR}{ext}

\makeatletter
\@namedef{subjclassname@1991}{\emph{2020} Mathematics Subject Classification}
\makeatother

\begin{document}

\title{Boundedness of Hyperbolic Varieties}

\author{Jackson S. Morrow}
\address{Jackson S. Morrow \\
	Department of Mathematics \\
	University of California, Berkeley \\
	810 Evans Hall, Berkeley, CA 94720}
\email{jacksonmorrow@berkeley.edu}

\begin{abstract}
Let $k$ be an algebraically closed field of characteristic zero, and let $X/k$ be a projective variety. 
The conjectures of Demailly--Green--Griffiths--Lang posit that every integral subvariety of $X$ is of general type if and only if $X$ is algebraically hyperbolic i.e., for any ample line bundle $\mathcal{L}$ on $X$ there is a real number $\alpha(X,\mathcal{L})$, depending only on $X$ and $\mathcal{L}$, such that for every smooth projective curve $C/k$ of genus $g(C)$ and every $k$-morphism $f\colon C\to X$,  $\text{deg}_Cf^*\mathcal{L} \leq \alpha(X,\mathcal{L})\cdot g(C) $ holds.

In this work, we prove that if $X/k$ is a  projective variety such that every integral subvariety is of general type,  then for every ample line bundle $\mathcal{L}$ on $X$ and every integer $g\geq 0$, there is an integer $\alpha(X,\mathcal{L},g)$, depending only on $X,\mathcal{L},$ and $g$, such that for every smooth projective curve $C/k$ of genus $g$ and every $k$-morphism $f\colon C\to X$, the inequality $\text{deg}_Cf^*\mathcal{L} \leq \alpha(X,\mathcal{L},g)$ holds, or equivalently, the Hom-scheme $\underline{\text{Hom}}_k(C,X)$ is projective. 
\end{abstract}

\dedicatory{Dedicated to my wife Lea for her unbounded love, support, and hyperbolicity}

\subjclass
{11G99 
(32Q45, 
14G05, 
14J29, 
14C05)} 

\keywords{Hyperbolicity, Demailly--Lang conjectures, non-Archimedean geometry, Hom-schemes}
\date{\today}
\maketitle


\section{\bf Introduction}
\label{sec:intro}
In 1922, Mordell \cite{Mordell1992rational} conjectured that for a smooth projective curve of genus $g\geq 2$ defined over a number field $F$,  its set of $F$-rational points is finite. 
This question guided research in arithmetic geometry for the next several decades and was first proved by Faltings \cite{Faltings2} in 1983. 
There are several other proofs of this conjecture due to Vojta \cite{Vojta:Mordell}, Bombieri \cite{Bombieri:Mordell},  and Lawrence--Venkatesh \cite{LawrenceVenkatesh:Mordell}, and even uniform versions by \cite{KRZB:Uniform} and \cite{DGH:Uniform, Kuhne:Equidistribution}.

In \cite{Lang:HyperbolicDiophantineAnalysis}, Lang proposed a higher dimensional analogue of Mordell's conjecture, and the most optimistic form of this conjecture states that for a smooth projective variety $X$ of general type (\autoref{defn:generaltype}) defined over a number field $F$, the set of $F$-rational points $X(F)$ is not Zariski dense; when $X$ is $1$-dimensional, this recovers Mordell's conjecture. 
By work of Faltings \cite{FaltingsLang1, FaltingsLang2}, Kawamata \cite{Kawamata}, and Ueno \cite{Ueno}, we know Lang's conjecture holds for closed subvarieties of abelian varieties, but outside of this setting, the conjecture remains wide open. 
This conjecture of Lang represents one of the deepest and longest standing questions in arithmetic geometry.

Instead of studying the arithmetic aspects of varieties of general type, one can analyze their geometric, complex analytic, and $p$-adic analytic properties and hope to use this information to prove results concerning their arithmetic. 
In all of these settings, there are conjectural characterizations due to Demailly \cite{Demailly}, Javanpeykar--Kamenova \cite{JKam}, Lang \cite{Lang:HyperbolicDiophantineAnalysis}, Green--Griffiths \cite{GreenGriffiths:Conj}, Cherry \cite{Cherry, CherryKoba},  Javanpeykar--Vezanni \cite{JVez}, Rosso and the author \cite{MorrowNonArchGGLV, MorrowRosso:Special} of being of general type, namely:~pseudo-algebraically hyperbolic, pseudo-bounded, pseudo-groupless (geometric); pseudo-Brody hyperbolic, pseudo-Kobayashi hyperbolic (complex analytic); pseudo-$K$-analytically Brody hyperbolic ($p$-adic analytic). 
The word ''pseudo'' appearing in each of these terms refers to the fact that there should exist some proper closed subvariety of $X$, the so called special locus, for which morphisms from certain kinds of objects factor through. 
Understanding the equivalence of these properties reduces to showing that various special loci coincide. 
We refer the reader to the survey \cite{Javanpeykar:Survey} for more details.

In this work, we study varieties such that every integral subvariety is of general type,  and we prove that such a variety is bounded (\autoref{defn:bounded}).  This result fits into the conjectures of Demailly and Lang as follows.

Let $L$ be an algebraically closed field of characteristic zero and let $X/L$ be a projective variety. 
Conjectures of Demailly and Lang claim that the condition that every integral subvariety of $X$ is of general type is equivalent to that condition that $X$ is algebraically hyperbolic i.e., for every ample line bundle $\mathcal{L}$ on $X$, there is a real number $\alpha({X,\mathcal{L}})$ depending only on $X$ and $\mathcal{L}$ such that for every smooth projective curve $C$ over $L$ and every $L$-morphism $f\colon C\to X$, the inequality $\deg_Cf^*\mathcal{L} \leq \alpha({X,\mathcal{L}})g(C)$ holds where $g(C)$ is the genus of $C$. 
By work of \cite{Bloch26, Ochiai77, Kawamata, Demailly} (respectively, \cite{Ein:SubvaritiesGeneral,Voisin:ConjectureClemens2, Pacienz:SubvarietiesGeneralType}), we know that this conjecture holds for a closed subvariety of an abelian variety (respectively, general hypersurfaces of sufficiently large degree). 
This conjecture also holds for projective varieties which admit a large local system by \cite{Brunebarbe:HyperbolicityLarge}. 

The main result of this paper states that if every integral subvariety of $X$ is of general type,  there exists an integer $\alpha(X,\mathcal{L},g(C))$ depending on $X,\cL,$ and $g(C)$, such that $\deg_Cf^*\mathcal{L} \leq \alpha(X,\mathcal{L},g(C)).$

\begin{thmx}\label{xthm:main1}
Let $L$ be an algebraically closed field of characteristic zero, and let $X/L$ denote a  projective variety. 
If every integral subvariety of $X$ is of general type, then for every ample line bundle $\mathcal{L}$ on $X$ and every integer $g\geq 0$, there is an integer $\alpha(X,\mathcal{L},g)$ such that for every smooth, projective curve $C/L$ of genus $g$ and every $L$-morphism $f\colon C\to X$, the inequality $\deg_Cf^*\mathcal{L} \leq \alpha(X,\mathcal{L},g)$ holds.
\end{thmx}

\autoref{xthm:main1} is significant progress towards the conjectures of Demailly and Lang.  
The discrepancy between our \autoref{xthm:main1} and the Demailly--Lang conjectures is the dependency on $g(C)$ in the bound; the conjectures require the bound to be linear in $g(C)$ whereas our result does not make any claims concerning the behaviour of $g(C)$ in the bound. 
While we do not prove the linearity of $g(C)$ in our bound, our result provides a bound \textit{uniform} in $g(C)$, and it applies to \textit{any}  projective algebraic variety for which every integral subvariety of $X$ is of general type. 
We also remark that \autoref{xthm:main1} represents progress towards the function field analogue of Vojta's conjecture (cf.~\cite[Conjecture 7]{RousseauTurchetWang:Nonspecial}), however we will not explain this relationship further. 

We record a few corollaries of \autoref{xthm:main1}. 
First,  we enhance \autoref{xthm:main1} to an equivalence in the setting where $X$ is a surface. 

\begin{corox}\label{xcoro:main1}
Let $L$ be an algebraically closed field of characteristic zero, and let $X/L$ denote a  projective variety of dimension at most two. Then the following are equivalent:
\begin{enumerate}
\item every integral subvariety of $X$ is of general type;
\item  for every ample line bundle $\mathcal{L}$ on $X$ and every integer $g\geq 0$, there is an integer $\alpha(X,\mathcal{L},g)$ such that for every smooth, projective curve $C/L$ of genus $g$ and every $L$-morphism $f\colon C\to X$, the inequality $\deg_Cf^*\mathcal{L} \leq \alpha(X,\mathcal{L},g)$ holds i.e., $X$ is $1$-bounded over $L$ (\autoref{defn:bounded});
\item $X$ is groupless over $L$ (\autoref{defn:groupless}). 
\end{enumerate}
\end{corox}

Second,  we show that Hom-schemes of hyperbolic varieties are projective schemes. 

\begin{corox}\label{xcoro:main2}
Let $L$ be an algebraically closed field of characteristic zero, and let $X/L$ be a  projective variety such that every integral subvariety is of general type. 
For any normal projective variety $Y/L$,  the Hom-scheme $\underline{\Hom}_L(Y,X)$ is projective. 
\end{corox}

Third,  we prove that a hyperbolic projective surface admits only finitely many non-constant maps defined over a fixed number field from a given variety. 

\begin{corox}\label{xcoro:main3}
Let $X$ be a projective surface over a finitely generated field $F$ of characteristic zero such that there is an embedding $F \hookrightarrow \mC$ with every integral subvariety of $X_{\mC}$ being of general type. Then, for every finitely generated extension $F'/F$ and every variety $V$ over $F'$, the set of non-constant rational maps $f\colon V \dashrightarrow X$ over $F'$ is finite.
\end{corox}

Fourth, we deduce finiteness of pointed morphisms into hyperbolic varieties, which can be viewed as a function field analogue of Lang's conjecture on finiteness of rational points over number fields for hyperbolic varieties. 

\begin{corox}\label{xcoro:main4}
Let $L$ be an algebraically closed field of characteristic zero, and let $X/L$ be a  projective variety such that every integral subvariety is of general type. 
For any normal variety $Y/L$ and point $y\in Y$ and every point $x\in X$, the set of maps $Y\to X$ sending $y\to x$ is finite.
\end{corox}

Finally, we have that the arithmetic persistence conjecture (\cite[Conjecture 1.5]{Javanpyekar:ArithmeticHyperbolicity}) holds for hyperbolic varieties. 

\begin{corox}\label{xcoro:main5}
The arithmetic persistence conjecture holds for every  projective variety such that every integral subvariety is of general type. 
\end{corox}

\subsection*{Main contributions}
Our proof of \autoref{xthm:main1} uses equi-characteristic non-Archimedean geometry, and the novel input is our construction and detailed study of a new, non-Archimedean analogue of the Kobayashi pseudo-metric for constant Berkovich analytic spaces.

To state these constructions, let $k$ be an algebraically closed, \textit{countable} field of characteristic zero,  let $K$ denote the completion of an algebraic closure of the formal Laurent series with coefficients in $k$ with respect to the $t$-adic valuation, let $X/k$ be a  projective variety, let $X_K$ denote the base change of $X$ to $K$, and let $X_K^{\an}$ denote the Berkovich analytification of $X_K$.  (Although there are various non-Archimedean analytifications, we work with Berkovich spaces due to their nice topological properties (cf.~\autoref{rem:WhyBerko})). 

In Section \ref{sec:pseudo}, we define a pseudo-metric $d_{X_K^{\an}}$ on $X_K^{\an}$, called the $K$-Kobayashi pseudo-metric, whose construction shares similarities to the classical Kobayashi pseudo-metric \cite{Kobayashi:Intrinsic}. However there are several delicate and technical issues which occur in order to achieve an actual pseudo-metric whose properties are similar to those of the classical Kobayashi pseudo-metric.  
Roughly, we say that $X_K^{\an}$ is $K$-Lea hyperbolic\footnote{The name $K$-Lea hyperbolic is an homage to my wife, Lea Beneish.  We refer the reader to the Acknowledgements for more details.}(or:~$K$-analytically Kobayashi hyperbolic) if the $K$-Kobayashi pseudo-metric is in fact a metric; we refer the reader to \autoref{defn:Leahyper} for a precise definition.

We justify the claim that our pseudo-metric behaves like the Kobayashi pseudo-metric through two results. 
First, we prove a non-Archimedean analogue of a theorem of Barth \cite{Barth:Standard} concerning the metrizability of a connected complex space with respect to the classical Kobayashi pseudo-metric. 

\begin{thmx}\label{xthm:LeaimpliesBarthintro}
 If $X_K^{\an}$ is $K$-Lea hyperbolic, then the $K$-Kobayashi metric defines the Berkovich $K$-analytic topology.   
\end{thmx}

\begin{remark}
We note that the assumption that $k$ is countable is necessary here as otherwise $X_K^{\an}$ need not be metrizable. We refer the reader to \cite{HruskovskiLoeserPoonen:BerkovichEmbed} and \cite{Favre:Countable} for further discussion of the metrizability of Berkovich spaces. 
\end{remark}

Next, we show that a $K$-Lea hyperbolic $K$-analytic space is $K$-analytically Brody hyperbolic, a notion which was first defined in \cite[Definition 2.3]{JVez}.  This is analogous to the statement that a Kobayashi hyperbolic space does not admit any non-constant entire curve. 

\begin{thmx}\label{xthm:LeaimpliesBrody}
If $X_K^{\an}$ is $K$-Lea hyperbolic, then $X_K^{\an}$ is $K$-analytically Brody hyperbolic i.e., for every connected algebraic group $G$ over $K$,  every morphism $G^{\an}\to X_K^{\an}$ is constant. 
\end{thmx}

For our purposes, the main attribute of our $K$-Kobayashi pseudo-metric is that the degeneracy of the pseudo-metric is measured by hyperbolic properties of $X$. 
In particular, the next result shows that $K$-Lea hyperbolicity can be inherited from hyperbolicity of $X$.

\begin{thmx}\label{xthm:HyperbolicImpliesLeaIntro}
If every integral subvariety of $X$ is of general type, then $X_K^{\an}$ is $K$-Lea hyperbolic. 
\end{thmx}

\begin{remark}\label{rem:WhyBerko}
Non-Archimedean analytic spaces play an important role in our work. 
While there are several notions of a non-Archimedean analytification \`a la Tate \cite{TateRigid}, Berkovich \cite{BerkovichEtaleCohomology}, or Huber \cite{huber2}, we use Berkovich spaces due to their nice topological properties.

In particular, we will need to consider continuous paths and metrics on these spaces, and we note that adic spaces do not admit non-constant continuous paths and are not Hausdorff. 
In short, the only category of non-Archimedean analytic spaces which posses the correct topological properties for us are Berkovich's analytic spaces. 
\end{remark}

\begin{remark}
In Section \ref{sec:pseudo}, we define our non-Archimedean Kobayashi pseudo-metric on constant $K$-analytic spaces i.e., the Berkovich analytic spaces $X_K^{\an}$ from the above theorem statements. 
While we only offer a definition for these $K$-analytic spaces, it will be clear to the reader how one could extend the definition to a connected, proper, reduced $K$-analytic space where $K$ is any algebraically closed, non-Archimedean valued field of characteristic zero which contains a countable dense subfield. 
We do not explore this construction in our current work, but we believe it to be of general interest. 
\end{remark}

\subsection*{Ideas of proof}
The core idea of the proof of \autoref{xthm:main1} is simple, and we describe it below. 

First, we recall the notation. 
Let $L$ be an algebraically closed field of characteristic zero,  let $X/L$ denote a  projective variety. 
From results in \cite{JKam}, the statement of \autoref{xthm:main1} is equivalent to showing that if every integral subvariety of $X$ is of general type, then for every smooth projective curve $C/L$, the Hom-scheme $\underline{\Hom}_L(C,X)$ is of finite type i.e., $X$ is bounded over $L$ (\autoref{defn:bounded}). 
Let $K$ denote the completion of an algebraic closure of the formal Laurent series with coefficients in $L$ with respect to the $t$-adic valuation, and let $X_K$ denote the base change of $X$ to $K$. 
By the geometric properties of boundedness (\autoref{thm:boundedgeometric}), we have that $X$ is bounded over $L$ if and only if $X_K$ is bounded over $K$. 
The advantage of working over $K$ is that we can consider the Berkovich analytification $X_K^{\an}$ of $X_K$. 
Next, we define a pseudo-metric on $X_K^{\an}$ such that when it is a metric, it defines the Berkovich $K$-analytic topology. 
We have that it is a metric when every integral subvariety of $X$ is of general type (cf.~ \autoref{xthm:LeaimpliesBarthintro} and \autoref{xthm:HyperbolicImpliesLeaIntro}). 
These results, the Arzela--Ascoli theorem, and additional arguments understanding uniform limits of $K$-analytic morphisms will help us to show that $X_K$ is bounded over $K$, and hence that $X$ is bounded over $L$. 
We note that Kobayashi \cite[Theorem 5.1.1]{Kobayashi} proved a variant of this result of Kobayashi hyperbolic varieties, and our above results allow us to imitate parts of his proof. 

In more detail, the proof of \autoref{xthm:main1} will follow in three steps. 
Recall that $k$ is an algebraically closed, \textit{countable} field of characteristic zero,  $K$ is the completion of an algebraic closure of the formal Laurent series with coefficients in $k$ with respect to the $t$-adic valuation,  $X/k$ is a projective variety, $X_K$ is the base change of $X$ to $K$, and $X_K^{\an}$ is the Berkovich analytification of $X_K$.  

\subsection*{Step 1:~Construction of a non-Archimedean Kobayashi pseudo-metric}
As mentioned above, the first input is the construction of a pseudo-metric $d_{X_K^{\an}}$ on $X_K^{\an}$. 
To do so, we identify a class of test objects, namely connected $K$-affinoid spaces, and construct a pseudo-metric using chains of these connected $K$-affinoids.  The morphisms in these chains are pointed at non-rigid points of $X_K^{\an}$ and satisfy certain algebraicity conditions. 
The construction of this pseudo-metric is not sufficient, and so we need to define another pseudo-metric which incorporates chains of closed constant $K$-subvarieties of $X_K^{\an}$.  
(For more motivation and details, see Section \ref{sec:pseudo}.)

\subsection*{Step 2:~Relating degeneracy of the pseudo-metric to hyperbolicity}
Once we have our pseudo-metric, we need to understand its behavior. 
The first result in this direction is \autoref{xthm:LeaimpliesBarthintro}, which shows that when the pseudo-metric is a metric, it defines the Berkovich $K$-analytic topology. 
Next, we prove that the pseudo-metric is not a metric on any projective variety which admits a polarized dynamical system (\autoref{thm:vanishingdynamical}), and this result allows us to deduce \autoref{xthm:LeaimpliesBrody}. 
On the other hand, using a deep result of Kobayashi--Ochiai \cite{KobayashiOchiai:MeromorphicMappings} (\autoref{thm:KobayashiOchiai}), we show in \autoref{xthm:HyperbolicImpliesLeaIntro} that the pseudo-metric is a metric when every integral subvariety of $X$ is of general type.

\subsection*{Step 3:~Studying limits of $K$-analytic morphisms}
The final step involves studying limits of $K$-analytic morphisms. 
More precisely, we prove that when $X$ is hyperbolic, our pseudo-metric will be distance decreasing with respect to $k$-analytic morphisms from analytifications of smooth projective constant curves $C_K^{\an}$ to $X_K^{\an}$ (note that here we mean the analytification of the base change of a $k$-morphism $C \to X$). 
With this result,  we will use the Arzela--Ascoli theorem (\autoref{thm:ArzelaAscoli}) to deduce that $\Hom_k^{\an}(C_K^{\an},X_K^{\an})$ is relatively compact in $\mathcal{C}(C_K^{\an},X_K^{\an})$. 
In \cite[Theorem 5.1.1]{Kobayashi},  Kobayashi noted that the space of holomorphic maps between complex manifolds is closed in the space of continuous map, and so the Arzela--Ascoli theorem implies his desired result.

In the non-Archimedean setting, it is not true that the space of $K$-analytic morphisms is closed in $\mathcal{C}(C_K^{\an},X_K^{\an})$ (see Subsection \ref{subsec:weaklyanalytic} for details), and so further arguments are required. 
In Section \ref{sec:proofmain},  we show that for every sequence $(f_n)$ of $k$-analytic maps,  there exists sub-sequence $(f_n')$ of $(f_n)$,  a complete field extension $K'/K$, and a $K'$-analytic morphism $F\colon C_{K'}^{\an} \to X_{K'}^{\an}$ such that $(f_n')$ converges to $\pi_{K'/K}(F)$ in $\underline{\Hom}_K(C_K,X_K)^{\an}$ where $\pi_{K'/K}$ is the base change morphism $\underline{\Hom}_K(C_K,X_K)^{\an}_{K'} \to \underline{\Hom}_K(C_K,X_K)^{\an}$. 
This result and deep topological properties of Berkovich $K$-analytic spaces proved by Poineau \cite{Poineau:AngelicBerkovich} allow us to deduce that $\Hom_k^{\an}(C_K^{\an},X_K^{\an})$ is relatively compact in $\underline{\Hom}_K^{\an}(C_K^{\an},X_K^{\an})$ (cf.~\autoref{prop:relativelycompact}). 
The final step is to use base change morphisms and the theory of Berkovich spaces over trivially valued fields to conclude that $\underline{\Hom}_k(C,X)$ is quasi-compact\footnote{We will reserve the term compact to mean quasi-compact and Hausdorff.}.

\begin{remark}[Choices involved in the pseudo-metric]
\label{rem:choiceinvolved}
The construction of our pseudo-metric involves various choices (see e.g., Sections \ref{sec:metrizability} and \ref{sec:pseudo}). 
From a purely aesthetic perspective, these choices are unsatisfactory. 
That being said, we believe that in order to construct a pseudo-metric on Berkovich $K$-analytic spaces, one must make some choices.

The main reason for this is that Berkovich $K$-analytic spaces are not locally modeled by a \textit{single} $K$-analytic object; one can see this immediately in the setting of $K$-analytic curves. This means that if one was to define a pseudo-metric using a single test object it could be not defined on \textit{every} point of \textit{every} Berkovich $K$-analytic space.  We remind the reader that a Berkovich $K$-analytic space contains many more points apart from the $K$-points. 

The choices in the construction allow us to define a pseudo-metric on every point of every constant $K$-analytic variety (i.e., a $K$-analytic space of the form $X_K^{\an}$ as above). 
We note that different choices will lead to different pseudo-metrics. 
However, our vanishing (\autoref{thm:vanishingdynamical}) and non-vanishing (\autoref{xthm:HyperbolicImpliesLeaIntro}) results are independent of these choices, in that they hold for any of the choices made in the construction. 
\end{remark}

\subsection*{Related Works}
Our work falls into the realm of non-Archimedean hyperbolicity, and although this area is only a few decades old, it has greatly matured over the past few years.

The study of non-Archimedean hyperbolicity was initiated in the thesis of W.~Cherry \cite{Cherry, CherryKoba} where he proved, amongst other things, a non-Archimedean analogue of the Bloch and Ochiai theorem. He also defined a variant of the Kobayashi pseudo-metric for Berkovich analytic spaces, and stated several conjectures, which have guided this area. 
We discuss the comparison of our non-Archimedean Kobayashi pseudo-metric and Cherry's pseudo-metric in \autoref{rem:compareCherry}.

Recently, Javanpeykar--Vezzani \cite{JVez} revisited Cherry's work and offered a definition of $K$-analytically Brody hyperbolic. In this work, the authors prove several properties of $K$-analytically Brody hyperbolic $K$-analytic spaces and discuss how hyperbolicity of the generic fiber can be inherited from the special fiber and vice versa (\textit{loc.~cit.~}Theorem 1.3). 
Work of the author and Rosso \cite{MorrowNonArchGGLV, MorrowRosso:Special} defined the notion of a pseudo-$K$-analytically Brody hyperbolic and proved a non-Archimedean analogue of a result of Noguchi \cite{Nogu}, which extended work of Bloch \cite{Bloch26} and Ochiai \cite{Ochiai77} to the semi-abelian setting; we note that the study of analytic morphisms from $\mG_{m,K}^{\an}$ into a semi-abelian variety was already discussed in \cite{AnCherryWang}.   

In \cite{Vazquez_HyperbolicityNotions, Vazquez_NormalFamily}, Rodr\'iguez V\'azquez proved a non-Archimedean version of Montel's theorem and showed that hyperbolicity of a curve (i.e., genus of the curve is $\geq 2$) is equivalent to the normalcy of the family of analytic morphisms from a basic tube (\autoref{defn:basictube}) of dimension 1 into the curve when the residue field is a countable, algebraically closed field of characteristic zero. 
In the non-Arichmedean setting, normalcy of a family of analytic morphisms is different from normalcy of a family of holomorphic morphisms in the complex analytic setting (see \cite[p.~1679]{Vazquez_NormalFamily} for definition). 
It would be interesting to investigate the relationship between normality and hyperbolicity further.

\subsection*{Organization}
In Sections \ref{sec:hyperbolicity}, \ref{sec:pseudomericback}, and \ref{sec:backgroundBerkovich}, we recall background on various notions of hyperbolicity, pseudo-metric induced topologies, length spaces, and Berkovich spaces, respectively. 
We note that Section \ref{sec:backgroundBerkovich} contains material on weakly analytic morphisms and non-Archimedean Hilbert schemes. 
In Section \ref{sec:metrizability}, we recall work concerning the metrizability of $K$-affinoid spaces. 

The new constructions of our work appear in the later sections. 
First,  in Section \ref{sec:pseudo}, we define our non-Archimedean Kobayashi pseudo-metric on a constant $K$-analytic variety $X_K^{\an}$ and define the notion of a $K$-Lea hyperbolic $K$-analytic space.

In Section \ref{sec:nonArchBarth}, we prove \autoref{xthm:LeaimpliesBarthintro}, which is a non-Archimedean analogue of a result of Barth, and next in Section \ref{sec:vanishing},  we describe the vanishing of our non-Archimedean Kobayashi pseudo-metric on a variety which admitting a polarized dynamical system and prove \autoref{xthm:LeaimpliesBrody}. This result may be viewed as the non-Archimedean analogue of several vanishing results concerning the Kobayashi pseudo-metric. 

The proof of our main theorems appear in the last two sections. 
In Section \ref{sec:proofLeaHyper}, we prove \autoref{xthm:HyperbolicImpliesLeaIntro} by showing that when every integral subvariety of $X$ is of general type,  $X_K^{\an}$ is $K$-Lea hyperbolic, and we conclude in Section \ref{sec:proofmain} with a proof of \autoref{xthm:main1} and the various corollaries by the methods outlined above. 

\subsection*{Acknowledgements}
This work owes much to many people.  First and foremost, I thank my wife, Lea Beneish. 
She has been a constant source of encouragement and support (and hyperbolicity!), and without her, this work would not exist. It is a privilege to dedicate this work to her.

I am also grateful to Matt Baker, Robert Benedetto, Morgan Brown, William Cherry, Brian Conrad, Henri Darmon, Alain Escassut, Ariyan Javanpeykar, Johannes Nicaise, Martin Olsson,  J\'er\^{o}me Poineau, Joseph Rabinoff, Ravi Vakil, Alberto Vezzani,  and David Zureick-Brown for engaging discussions on aspects of this result ranging from $p$-adic dynamical systems to Berkovich spaces to Hom-schemes. 
I am especially thankful to Ariyan Javanpeykar for impressing upon me the importance of the construction of a non-Archimedean Kobayashi pseudo-metric and describing how this construction would have implications to the conjectures of Demailly and Lang. 
I would also like to thank Lea Beneish, Sebastien Boucksom, Sean Howe, Ariyan Javanpeykar, Mattias Jonsson, J\'er\^{o}me Poineau, and Eric Riedl for helpful comments on a first draft. 

During the preparation of this article,  the author was partially supported by NSF RTG grant DMS-1646385 and later supported by NSF MSPRF grant DMS-2202960.

\section{\bf Conventions}
\label{sec:conventions}
We establish the following conventions throughout.

\subsection{Fields}\label{subsec:fields}
We let $k$ be a \textit{countable}, algebraically closed field of characteristic zero. 
We will also let $L$ denote any algebraically closed field of characteristic zero.  
Finally, we let $K$ denote the completion of an algebraic closure of the formal Laurent series with coefficients in $k$ with respect to the $t$-adic valuation. 
The field $K$ is an algebraically closed, complete non-Archimedean, non-trivially valued field of characteristic zero which is separable i.e., it contains a countable dense subset, namely an algebraic closure $\overline{k(t)}$ of $k(t)$.

\subsection{Varieties}\label{subsec:var}
By a variety over $L$, we mean a separated, integral, positive dimensional scheme of finite type over $\Spec(L)$. 
A curve (resp.~surface) over $L$ is a $1$-dimensional (resp.~$2$-dimensional) variety.  
For an extension $L'/L$ and variety $X/L$,  we will let $X_{L'}$ denote the base change of $X$ to $L'$. 
When $L = k$ and $L' = K$ as above, we will say that $X_K$ is the \cdef{constant variety} associated to $X$. 
For a subset $U$ of a variety $X/L$, we let $\overline{U}$ denote the Zariski closure of $U$ in $X$.

By a subvariety of a variety over $L$,  we mean a positive dimensional closed reduced subscheme.  Note that a subvariety is determined by its underlying topological space and we frequently identify both.
By a point of a variety over $L$, we mean a $0$-dimensional closed integral subscheme. 
For a  variety $X/k$ and a line bundle $\cL$ on $X$, we use $\kappa(X,\cL)$ to denote the Iitaka dimension of $\cL$, which we recall. 
For each $m\geq 0$ such that $h^0(X,\cL^{\otimes m}) \neq 0$, the linear system $|\cL^{\otimes m}|$ induces a rational map from $X$ to a projective space of dimension $h^0(X,\cL^{\otimes m}) - 1$. 
The Iitaka dimension of $\cL$ is the maximum over all $m\geq 1$ of the dimension of the image for this rational map. 
The Kodaira dimension $\kappa(X)$ of $X$ is the Iitaka dimension of the canonical bundle $K_X$.

\subsection{Hom-schemes}\label{subsec:homschemes}
We will make extensive use of Hom-schemes in this paper, and so we recall some of their basic properties. Let $S$ be a scheme, and let $X\to S$ and $Y\to S$ be projective flat morphisms of schemes. 
By Grothendieck's theory of Hibert and Quot schemes \cite{GrothendieckHilbertSchemes}, the functor
\[
\text{Sch}/S^{\text{op}} \to \text{Sets}, \quad (T\to S) \mapsto \Hom_T(Y_T,X_T)
\]
is representable by a locally of finite type, separated $S$-scheme, which we denote by $\underline{\Hom}_S(X,Y)$.  
Moreover, for $h\in \mQ[t]$, the subfunctor parametrizing morphisms whose graph has Hilbert polynomial $h$ is representable by a quasi-projective subscheme $\underline{\Hom}_S^{h}(X,Y)$ of $\underline{\Hom}_S(X,Y)$. We can write $\underline{\Hom}_S(X,Y) = \bigsqcup_{h\in \mQ[t]}\underline{\Hom}_S^{h}(X,Y)$ i.e., as a disjoint union of quasi-projective $S$-schemes.  

\subsection{Non-Archimedean analytic spaces}\label{subsec:nonArch}
We will work with non-Archimedean analytic spaces in the sense of Berkovich \cite{BerkovichSpectral}. All of our Berkovich analytic spaces will be good i.e., admit an affinoid neighborhood. Throughout, we refer to a Berkovich space over $K$ as a \cdef{$K$-analytic space}. 
We will use script letters $\sX$, $\sY$, $\sZ$ to denote $K$-analytic spaces, which are not necessarily algebraic, and we will denote the $n$-dimensional open unit polydisk over $K$ (resp.~the $n$-dimensional closed unit polydisk over $K$) by $B^n(0,1)^-$ (resp.~$B^n(0,1)$).  For $x\in X_K^{\an}$, let $\sH(x)$ denote the completed residue field of $x$. 

In this work,  we will be primarily interested in the following kind of $K$-analytic space.

\begin{definition}\label{defn:constantKanalyticvariety}
Let $X$ be a proper variety over an algebraically closed field $k$, let $K$ be as above, and let $X_K$ denote the base change of $X$ to $K$. 
The $K$-analytic space $X_K^{\an}$ associated to $X_K$ will be referred to as the \cdef{constant $K$-analytic variety associated to $X$}. In most situations, we will simply say $X_K^{\an}$ is a \cdef{constant $K$-analytic variety}.  
The Berkovich space $X_K^{\an}$ is reduced, Hausdorff, boundaryless, path-connected, and compact by \cite[Theorem 3.4.8]{BerkovichSpectral}.  We recall these notions in Section \ref{sec:backgroundBerkovich}.  
Given a constant $K$-analytic variety $X_K^{\an}$, we define a \cdef{constant $K$-analytic subvariety} of $X_K^{\an}$ to be a $K$-analytic subspace of the form $Z_K^{\an}$ where $Z$ is a subvariety of $X$. 
\end{definition}

\section{\bf Notions of hyperbolicity}
\label{sec:hyperbolicity}
In this section, we define various notions of hyperbolicity which appear in the work, describe their basic properties, and recall results relating these notions.  To conclude, we state the (weak forms) of the conjecture of Demailly, Green--Griffiths, Javanpeykar--Kamenova, and Lang. 

\subsection{Algebraically hyperbolic and bounded varieties}
We begin with Demailly's \cite{Demailly} notion of algebraically hyperbolic. 
Recall that $L$ is an algebraically closed field of characteristic zero. 

\begin{definition}\label{defn:algebraicallyhyperbolic}
Let $X/L$ be a projective variety. 
We say that $X$ is \cdef{algebraically hyperbolic over $L$} if for every ample line bundle $\mathcal{L}$ on $X$, there is a real number $\alpha({X,\mathcal{L}})$ depending only on $X$ and $\mathcal{L}$ such that for every smooth projective curve $C/L$ and every $L$-morphism $f\colon C\to X$, the inequality $\deg_Cf^*\mathcal{L} \leq \alpha({X,\mathcal{L}})g(C)$ holds where $g(C)$ is the genus of $C$. 
\end{definition}

\begin{example}\label{exam:alghyp}
Using the Riemann--Hurwitz formula, it is clear that a projective curve $C$ is algebraically hyperbolic if and only if the genus $g $ of $C$ satisfies $g\geq 2$. 
In higher dimensions,  we know less about algebraically hyperbolic varieties, however we do have many partial results for hypersurfaces, which we briefly summarize.

It is well-known that hypersurfaces of degree less than $2n -2$ in $\mP^n$ will always contain lines and hence are not algebraically hyperbolic. 
Voisin \cite{Voisin:ConjectureClemens2} proved that a very general hypersurface in $\mP^n$ of degree $\geq 2n-1$ for $n\geq 4$ is algebraically hyperbolic. This result leaves one with determining the algebraically hyperbolic hypersurfaces in $\mP^3$ and the hypersurfaces of degree $2n-2$ which are algebraically hyperbolic.  
For $n = 3$, Xu \cite{Xu:SubvarietiesGeneralHypersurfaces} showed that very general hypersurfaces of degree $\geq 6$ in $\mP^3$ are algebraically hyperbolic, and recent work of Coskun--Riedl \cite{CoskunRiedl:AlgHyperGeneralQuintic} proved that a very general quintic surface in $\mP^3$ is algebraically hyperbolic. To summarize, we have a complete classification of very general hypersurfaces in $\mP^3$ which are algebraically hyperbolic.  
For hypersurfaces of degree $2n-2$, we note that for $n = 3$ these are K3 surfaces which contain rational curves and hence cannot be algebraically hyperbolic. 
For $n\geq 6$, Pacienza \cite{Pacienz:SubvarietiesGeneralType} proved that a very general hypersurface of degree $2n-2$ is algebraically hyperbolic for $n\geq 6$. 
In recent work \cite{Yeong:AlgHyp5}, Yeong proved that a very general hypersurface of degree $2n-2$ in $\mP^5$ is algebraically hyperbolic leaving the case of $n = 4$ the only open case. 
\end{example}

In \cite{JKam}, Javanpeykar and Kamenova defined a weaker notion of algebraically hyperbolic, which is related to topological properties of Hom-schemes. 

\begin{definition}\label{defn:bounded}
Let $X/L$ be a projective variety and let $n\geq 1$ be an integer. 
We say that $X$ is \cdef{$n$-bounded over $L$} if for every normal projective variety $Y/L$ of dimension at most $n$, the scheme $\underline{\Hom}_L(Y,X)$ is of finite type over $L$, in particular $\underline{\Hom}_L(Y,X)$ is quasi-compact. 
We say that $X$ is \cdef{bounded over $L$} if $X$ is $n$-bounded over $L$ for every $n\geq 1$. 
\end{definition}

The relationship between algebraically hyperbolic, $1$-bounded, and bounded is summarized in the following theorems of Javanpeykar and Kamenova. 

\begin{theorem}[\protect{\cite[Theorem 9.3 and Theorem 1.14]{JKam}}]\label{thm:equivalentbounded}
Let $X/L$ be a projective variety. Then, the following are equivalent:
\begin{enumerate}
\item $X$ is $1$-bounded over $L$,
\item $X$ is bounded over $L$,
\item For every ample line bundle $\mathcal{L}$ and every integer $g\geq 0$, there is an integer $\alpha(X,\mathcal{L},g)$ such that, for every smooth projective curve $C$ of genus $g$ over $L$ and every $L$-morphism $f\colon C\to X$, the inequality $\deg_Cf^{*}\mathcal{L} \leq \alpha(X,\mathcal{L},g)$ holds. 
\end{enumerate}
\end{theorem}

By \autoref{thm:equivalentbounded}, it is clear that being algebraically hyperbolic implies boundedness, and it is conjectured (cf.~\autoref{conj:DGGJKL}) that these notions are equivalent. 
The property of being algebraically hyperbolic and bounded are geometric i.e., they persist over any algebraically closed extension. 

\begin{theorem}[\protect{\cite[Theorem 7.1 \& Theorem 1.11]{JKam}}]\label{thm:boundedgeometric}
Let $X/L$ be a projective variety, and let $L'$ be an algebraically closed extension containing $L$. Then, $X$ is algebraically hyperbolic (resp.~bounded) over $L$ if and only if $X_{L'}$ is algebraically hyperbolic (resp.~bounded) over $L'$. 
\end{theorem}

\subsection{Groupless varieties}
Next, we define the notion of a groupless variety from \cite{Lang:HyperbolicDiophantineAnalysis, VojtaLangExc, JKam}. 

\begin{definition}\label{defn:groupless}
Let $X/L$ be a variety. We say that $X$ is \cdef{groupless over $L$} if for every finite type connected group scheme $G/L$, every $L$-morphism $G\to X$ is constant. 
\end{definition}

We recall useful properties of being groupless.

\begin{lemma}[\protect{\cite[Lemma 2.3, Lemma 2.7, \& Proposition 4.4]{JKam}}]\label{lemma:grouplessequivgeometric}
Let $X/L$ be a variety.
\begin{enumerate}
\item If $L'/L$ is an extension of algebraically closed fields of characteristic zero, then $X$ is groupless over $L$ if and only if $X_{L'}$ is groupless over $L'$. 
\item If $X/L$ is proper, then $X$ is groupless over $L$ if and only if for every abelian variety $A$ over $L$, every $L$-morphism $A\to X$ is constant.
\item If $X/L$ is projective and bounded, then $X$ is groupless over $L$. 
\end{enumerate}
\end{lemma}

Using \autoref{lemma:grouplessequivgeometric} and \autoref{exam:alghyp}, we can find various examples of groupless varieties.

\subsection{Varieties of general type}
Continuing our discussion on hyperbolicity,  we recall basic definitions of varieties of general type and some deep results concerning maps to varieties of general type.

\begin{definition}\label{defn:generaltype}
Let $X/L$ be a projective variety. We say that $X$ is \cdef{of general type} if it has a desingularization $X'\to X$ with $X'/L$ a smooth projective variety such that the canonical bundle $\omega_{X'/L}$ is a big line bundle. 
\end{definition}

\begin{example}
A smooth curve $C/L$ is of general type (and hence hyperbolic) if and only if the genus $g(C)$ is greater than one.  
Using additivity of Kodaira dimension, we have that the $n$-fold product $C^n := C \times \cdots \times C$ is of general type if and only if $C$ is of general type.  
\end{example}

\begin{lemma}\label{lemma:genearltypegeometric}
Let $X/L$ be a projective variety which is of general type. 
If $L'/L$ is an extension of algebraically closed fields of characteristic zero, then $X_L$ is of general type. 
\end{lemma}

\begin{proof}
This follows from comparing the global sections of $\omega_{X/L}$ and $\omega_{X_L'/L'}$ using base change arguments. 
\end{proof}

\begin{definition}\label{defn:hyperbolic}
Let $X/L$ be a projective variety. We say that $X$ is \cdef{hyperbolic over $L$} (or:~\cdef{hereditarily of general type over $L$}) if every integral subvariety of $X$ is of general type. 
\end{definition}

\begin{example}\label{exam:ample}
If $X/L$ is a smooth projective variety with ample cotangent bundle (i.e.,  $\cO_{\mP(\Omega_X^1)}(1)$ is ample), then $X$ is hyperbolic over $L$ (see e.g., \cite[Example 6.3.28]{Lazarsfeld:Positivity2}).  
We refer the reader to \cite[pages 39 -- 43]{Lazarsfeld:Positivity2}, \cite{Brobek:Thesis}, and \cite{BrotbekDarondeau:CompleteIntersectionsAmple} for further examples of varieties with ample cotangent bundle and for methods of constructing such varieties. 
\end{example}

\begin{example}\label{exam:hyperbolicnotample}
By \cite{Ein:SubvaritiesGeneral}, a very general hypersurface of degree at least $2n-1$ in $\mP^{n}$ is hyperbolic, and work of \cite{Pacienz:SubvarietiesGeneralType} proves that a very general hypersurface of degree at least $2n-2$ in $\mP^n$ where $n\geq 6$ is hyperbolic.  
By work of Schneider \cite{Schneider:Symmetric}, we remark that a smooth hypersurface in $\mP^n$ for $n\geq 3$ cannot have ample cotangent bundle, and so the previously mentioned results give examples of hyperbolic varieties which do not have ample cotangent bundle. 
\end{example}

Next, we record a lemma asserting that hyperbolic implies groupless and how hyperbolicity behaves under base change. 

\begin{lemma}\label{lemma:hyperbolicimpliesgroupless}
Let $X/L$ be a  projective variety. 
If $X$ is hyperbolic over $L$, then $X$ is groupless over $L$.
\end{lemma}

\begin{proof}
By our definition of hyperbolic and \autoref{lemma:grouplessequivgeometric}.(2), it suffices to prove that an abelian variety $A/L$ cannot dominant a positive-dimensional projective variety of general type. 
As such, we may reduce to showing that there does not exist a dominant morphism $f\colon A \to X$. 
By replacing $f$ by a birational model of the Stein factorization, we may assume that $f'\colon A'\to X'$ is a projective morphism of smooth varieties with $f_*'\cO_{A'} = \cO_{X'}$, $X'$ is of general type, and $A'\to A$ is birational. Let $F$ be the general fiber, then the Kodaira dimension $\kappa(F) \geq 0$ since $K_{A'|F}\cong K_{F}$. 
By \cite[Theorem 3]{Kawamata:Characterization},  we have that $\kappa(A') \geq \kappa(F) + \kappa(X')$, but this would imply that $\dim(X) = \kappa(X') = 0$, which is a contradiction. 
\end{proof}

\begin{lemma}\label{lemma:basechangehyper}
Let $X/L$ be a  projective hyperbolic variety.  
For every algebraically closed field extension $L'/L$,  every integral subvariety of $X_{L'}$ admits a dominant morphism to a variety of general type defined over $L'$. 
\end{lemma}

\begin{proof}
Let $Z'$ be an integral subvariety of $X_{L'}$. 
If $Z'$ is of the form $Z_{L'}$ where $Z$ is an integral subvariety of $X$, then the result follows from \autoref{lemma:genearltypegeometric}. 
Otherwise, consider the projection map $\pi_{L'/L}\colon X_{L'} \to X,$ which is quasi-compact by \cite[\href{https://stacks.math.columbia.edu/tag/01K5}{Tag 01K5}]{stacks-project}. 
Consider the set-theoretic image $\pi_{L'/L}(Z')$ and the reduced induced scheme structure $Z$ on $\pi_{L'/L}(Z')$.  Note that $Z$ is an integral subvariety of $X$. 
Since $\pi_{L'/L}$ is quasi-compact,  \cite[\href{https://stacks.math.columbia.edu/tag/01R8}{Tag 01R8}]{stacks-project} tells us that $\pi_{L'/L}(Z')$ is dense in $Z$, and since $Z'$ is a closed subvariety of $X_{L'}$, we have that the restriction $\pi_{L'/L \vert Z'}\colon Z' \to Z$ is a dominant morphism.
After pulling back along $\Spec(L')\to \Spec(L)$, we have a dominant morphism $Z' \to Z_{L'}$, where $Z_{L'}$ is of general type by \autoref{lemma:genearltypegeometric}. 
\end{proof}

\subsection{Dominant morphisms to varieties of general type}
Varieties of general type enjoy many finiteness results related to dominant rational maps, which we now recall.

The classical work of de Franchis and Severi states that for a smooth hyperbolic curve $C/L$ and a given projective curve $C'/L$, there are only finitely many non-constant (hence surjective) morphisms $C'\to C$, and for a fixed projective curve $C'/L$, there are only finitely many smooth projective hyperbolic curves $C/L$ for which there exists a non-constant morphism $C'\to C$.

There are higher dimensional versions of these results.  For the first statement, Kobayashi and Ochiai proved the following.

\begin{theorem}[\protect{\cite[Theorem 1]{KobayashiOchiai:MeromorphicMappings}}]\label{thm:KobayashiOchiai}
Let $X/L$ be a projective variety of general type. 
Then, for every projective variety $Y/L$, the set of dominant rational maps $f\colon Y \dashrightarrow X$ is finite.
\end{theorem}

The higher dimensional version of the second statement was known as the Severi--Iitaka conjecture, which was proved through the work of Maehara \cite{Maehara:FinitenessProperty} and Hacon--McKernan \cite{HaconMcKernan:BoundednessPluricanonical}, Tsuji \cite{Tsuji:Pluricanonical1, Tsuji:Pluricanonical2}, and Takayama \cite{Takayama:Pluriccanonical}.  

\begin{theorem}\label{thm:SeveriIitaka}
For any fixed variety $X/L$, there exist only finitely many dominant rational maps $X \dashrightarrow Y$ (modulo birational equivalence) where $Y/L$ is a variety of general type. 
\end{theorem}

\subsection{The conjectures of Demailly, Green--Griffiths, Javanpeykar--Kamenova, and Lang}
To conclude this section, we state the conjecture of Demailly \cite{Demailly}, Green--Griffiths \cite{GreenGriffiths:Conj}, Javanpeykar--Kamenova \cite{JKam}, and Lang \cite{Lang:HyperbolicDiophantineAnalysis} which posit that the various notions of hyperbolicity we defined above are in fact all equivalent. 

\begin{conjecture}[Demailly, Green--Griffiths, Javanpeykar--Kamenova, and Lang]\label{conj:DGGJKL}
Let $X/L$ be a  projective variety. Then, the following are equivalent:
\begin{enumerate}
\item $X$ is algebraically hyperbolic over $L$ (\autoref{defn:algebraicallyhyperbolic}),
\item $X$ is bounded over $L$ (\autoref{defn:bounded}),
\item $X$ is groupless over $L$ (\autoref{defn:groupless}),
\item $X$ is hyperbolic over $L$ (\autoref{defn:hyperbolic}). 
\end{enumerate}
\end{conjecture}

From our above discussion, we have the following immediate implications of \autoref{conj:DGGJKL}:
\[
\text{(1)}\Rightarrow\text{(2)}\Rightarrow\text{(3}) \text{ and }\text{(4)}\Rightarrow\text{(3)}. 
\]
By work of \cite{Bloch26, Ochiai77, Kawamata, Demailly}, we know that \autoref{conj:DGGJKL} holds for a closed subvariety of an abelian variety, and the references from \autoref{exam:alghyp} and \autoref{exam:hyperbolicnotample} imply that this conjecture holds for hypersurfaces of sufficiently large degree.

Our \autoref{xthm:main1} proves that (4)$\Rightarrow$(2) in full generality i.e.,  a  projective hyperbolic variety is bounded. 

\begin{remark}
\autoref{conj:DGGJKL} is known as the weak form of the conjectures of Demailly, Green--Griffiths, Javanpeykar--Kamenova, and Lang. The strong form of this conjecture predicts that the appropriate ``pseudoifications'' the above notions are equivalent (cf.~Section \ref{sec:intro}), and the strongest for of this conjecture describes certain exceptional loci related to each of these above and claims that these
exceptional loci are all equal. We refer the reader to \cite{Javanpeykar:Survey} for further discussion. 
\end{remark}

\section{\bf  Metric and topological preliminaries}
\label{sec:pseudomericback}

In this section, we discuss background on pseudo-metrics, metric induced topologies,  inner metric spaces, and a form of the Arzela--Ascoli theorem. 

\subsection{Pseudo-metrics and metric induced topologies}
We follow the exposition from \cite[Chapter 1]{Kobayashi}.  
These definitions and results will be used in our proof of \autoref{xthm:LeaimpliesBarthintro}. 
First, we recall the definition of a pseudo-metric, which is a non-negative real-valued function that behaves like a metric but it does not necessarily distinguish between two distinct points. 

\begin{definition}
A \cdef{pseudo-metric space} is a pair $(M,d)$ where $M$ is a set together with a non-negative real-valued function $d\colon M\times M \to \mR_{\geq 0}$ (called a \cdef{pseudo-metric}) such that for every $x,y,z\in M$:
\begin{enumerate}
\item $d(x,x) = 0$,
\item $d(x,y) = d(y,x)$,
\item $d(x,z) \leq d(x,y) + d(y,z)$.
\end{enumerate}
\end{definition}

For $(M,d)$ a pseudo-metric space where $M$ is a topological space, one can show that the pseudo-metric $d$ will induce the given topology on $M$ in certain situations, which we describe below. 

\begin{definition}\label{defn:rectifiable}
Let $M$ be an path-connected Hausdorff topological space with pseudo-metric $d$, which is continuous on $M$.
Given a path $\gamma(t)$, $0\leq t \leq 1$ in $M$, the \cdef{length} $L(\gamma)$ of $\gamma$ is defined by
\[
L(\gamma) = \sup \sum_{i=1}^k d(\gamma(t_{i-1}),\gamma(t_i))
\]
where the supremum is taken over all partitions $0 = t_0 < t_1 < t_2 < \cdots < t_k = 1$ of the interval $[0,1]$. 
A path is said to be \cdef{$d$-rectifiable} if its length $L(\gamma)$ is finite. 
\end{definition}

\begin{definition}\label{defn:innerpseudo}
We define a new pseudo-metric $d^i$ called the \cdef{inner pseudo-distance} induced by $d$ by setting
\[
d^i(x,y) = \inf L(\gamma)
\]
where the infimum is taken over all $d$-rectifiable paths $\gamma$ joining $x$ and $y$. 

Clearly, we have that 
\[
d(x,y)\leq d^i(x,y)
\]
for all $x,y \in M$, and we say that a pseudo-metric is \cdef{inner} if $d^i = d$. 
\end{definition}

\begin{prop}[\protect{\cite[Proposition 1.1.18.(3)]{Kobayashi}}]\label{prop:innermetrictopology}
Let $M$ be a locally compact, path-connected, Hausdorff topological and $d$ an inner pseudo-metric on $M$, which is continuous on $M$.
If $d$ is an inner metric, then $d$ induces the given topology of $M$.
\end{prop}

\subsection{Inner metric spaces}
We will use \autoref{prop:innermetrictopology} to prove \autoref{xthm:LeaimpliesBarthintro}.  
To use this result, we will need to understand when inner metrics exist on our test objects. 
First, we recall the definition of an inner metric. 

\begin{definition}
Let $M$ be an path-connected Hausdorff topological space with metric $d$, which is continuous on $M$. 
If $d$ is inner (i.e., $d^i = d$), then we say that $(M,d)$ is an \cdef{inner metric space}. 
\end{definition}

In general, a metric space $(M,d)$ need not be an inner metric. 
However, in favorable situations, one can find another metric $d'$ which preserves the original metric topology for which the new metric space $(M,d')$ is inner and every pair of distinct points in $(M,d')$ can be joined by a minimizing geodesic.  
%

\begin{theorem}\label{thm:affinoidsgeodesic}
A compact, path-connected, locally path-connected, locally compact metric space $(M,d)$ admits a topology preserving metric $d'$ such that $(M,d')$ is an inner metric space where any pair of distinct points in $M$ can be joined by a minimizing geodesic with respect to $d$. 
\end{theorem}

\begin{proof}
By (independent) results of Bing \cite[Theorem 8]{Bing:partitioning} and Moise \cite[Theorem 4]{Moise:Grille},  a compact, path-connected,  and locally path-connected metric space $(M,d)$ admits a topology preserving convex metric.
A result of Menger \cite[Corollary 2.4.12 \& Theorem 2.4.16]{BuragoBuragoIvanov:CourseMetric} tells us that a complete (hence compact) and convex metric space is an inner metric space, and to conclude, the Hopf--Rinow theorem \cite[Theorem 2.5.3]{BuragoBuragoIvanov:CourseMetric} states that for a complete (hence compact) and locally compact inner metric space, any two points can be connected by a minimizing geodesic. 
\end{proof}

%
%
%
%
%
%
%

\subsection{Mappings into metric spaces}
By virtue of \autoref{xthm:LeaimpliesBarthintro}, \autoref{xthm:main1} is a statement concerning analytic maps into a metric space. Results concerning such maps can be direct consequences of purely topological results about the space of maps into metric spaces. 
We recall the setup and statement of the Arzela--Ascoli theorem. 

Given two topological spaces $M,N$, we denote by $\mathcal{C}(M,N)$ the space of all continuous maps $f\colon M\to N$ equipped with the compact-open topology. If $N$ is a metric space, then this topology coincides with the topology of uniform convergence on compact sets, and if further $M$ is compact, this topology coincides with the topology of uniform convergence.

Suppose that $M,N$ are locally compact, separable, Hausdorff spaces with pseudo-metrics $d_M$ and $d_N$, respectively. 
Under these conditions, $\mathcal{C}(M,N)$ is second countable, and hence compactness is equivalent to sequential compactness. Let $\mathcal{D}(M,N) \subseteq \mathcal{C}(M,N)$ denote the family of distance-decreasing maps i.e., maps such that 
\[
d_N(f(x),f(y))\leq d_M(x,y)
\]
for all $x,y\in M$. 
Note that $\mathcal{D}(M,N)$ is closed in $\mathcal{C}(M,N)$. 

Below, we state a variant of the Arzela--Ascoli theorem.

\begin{theorem}[\protect{Arzela--Ascoli Theorem \cite[Corollary 1.3.2]{Kobayashi}}]\label{thm:ArzelaAscoli}
Let $M$ be a locally compact, separable, Hausdorff space with a pseudo-metric $d_M$, and let $N$ be a locally compact, separable metric space with metric $d_N$. 
A subfamily $\mathcal{F} \subseteq \mathcal{D}(M,N)$ is relatively compact in $\mathcal{C}(M,N)$ if and only if for every $x\in M$, the set $\brk{f(x): f\in \mathcal{F}}$ is relatively compact in $N$. 
\end{theorem}

\section{\bf Preliminaries on Berkovich $K$-analytic spaces}
\label{sec:backgroundBerkovich}
In this section, we recall background on Berkovich $K$-analytic spaces,  weakly analytic maps, and non-Archimedean Hilbert spaces. 
Recall our conventions established in Section \ref{sec:conventions}.  
We note that the material in this section holds when $K$ is an arbitrary algebraically closed, non-Archimedean, non-trivially valued field of characteristic zero. 

\subsection{Basic definitions}
To begin, we recall the types of Berkovich spaces we will be working with.  We refer the reader to Berkovich's originial texts \cite{BerkovichSpectral, BerkovichEtaleCohomology} and Temkin \cite{Temkin:IntroBerk} for the standard definitions of Berkovich spaces. 

\begin{definition}\label{defn:boundaryless}
A $K$-analytic space $\sX$ is called \cdef{good} if each point $x\in \sX$ admits an affinoid neighborhood $M(A)$, and $\sX$ is called \cdef{boundaryless} if each point $x\in \sX$ admits an affinoid neighborhood $M(A)$ such that $x$ lies in the interior of $M(A)$ i.e., $x\in \Int(M(A))$. 
A $K$-analytic space is called \cdef{proper} if $\sX$ is boundaryless and compact. 
\end{definition}

\begin{example}\label{exam:boundary}
For a scheme $X/K$ locally of finite type, we can associated to it the Berkovich analytification $X^{\an}$, which will always be a boundaryless $K$-analytic space. By \cite[Theorem 3.4.8]{BerkovichSpectral}, $X^{\an}$ is proper if and only if $X$ is proper. 
\end{example}

We will be interested in analytic morphisms between good $K$-analytic spaces.

\begin{definition}\label{defn:analyticmorphism}
Let $\sX$ and $\sY$ be good $K$-analytic spaces. 
A continuous map $f\colon \sX \to \sY$ is \cdef{$K$-analytic} if there exists an atlas $\brk{(\sU_i),\varphi_i}$ of $\sX$ and an atlas $\brk{(\sV_j,\psi_j)}$ of $\sY$ such that every pair of indices $i,j$, the composition $\psi_j\circ f\circ \varphi_i^{-1}\colon \varphi_i(\sU_i) \to \psi_j(\sV_j)$ satisfies the following property:~for every affinoid domain $\sW\subset \varphi_i(\sU_i)$ and $\sZ\subset \psi_j(\sV_j)$ with $\psi_j\circ f\circ \varphi_i^{-1}(\sW)$ contained in the topological interior of $\sZ$ in $\psi_j(\sV_j)$, the restriction of $\psi_j\circ f\circ \varphi_i^{-1}$ to $\sW$ is morphism of $K$-affinoid spaces i.e., is induced by a bounded morphism between underlying $K$-affinoid algebras. 
\end{definition}

\subsection{Local properties of Berkovich $K$-analytic spaces}
\label{sec:localproperties}
Next, we discuss local properties of boundaryless Berkovich spaces. 

Given a $K$-affinoid space $M(A)$, there exists a reduction map $\red_{M(A)}\colon M(A) \to \Spec(\wt{A})$ where $\wt{A} \coloneqq A^{\circ}/A^{\circ\circ}$ i.e., the power bounded elements of $A$ modulo the topologically nilpotent elements of $A$ (see \cite[Section 2.4]{BerkovichSpectral} for details). 
A standard fact is that the interior of $M(A)$ is isomorphic to the pre-image of the closed points of $\Spec(\wt{A})$ under the reduction map (see e.g., \cite[Sections 2.5 -- 2.6]{BerkovichSpectral}).   
A important type of $K$-analytic space will be those coming from pre-images of a single closed point.

\begin{definition}\label{defn:basictube}
A $K$-analytic space $\mathscr{U}$ is called a \cdef{basic tube} if there exists an equidimensional reduced strictly $K$-affinoid space $M(A)$ and a closed point $\wt{x}$ in the reduction of $M(A)$ such that $\sU \cong \red_{M(A)}^{-1}(\wt{x})$. 
\end{definition}

We record some useful properties of basic tubes. 

\begin{lemma}\label{lemma:equivalentbasictube}
Let $\mathscr{U}$ be a basic tube. 
\begin{enumerate}
\item $\mathscr{U}$ is a path-connected, boundaryless, $\sigma$-compact $K$-analytic space. 
\item A $K$-analytic space is a basic tube if and only if it is a connected component of the interior of some equi-dimensional, reduced strictly $K$-affinoid space.
\item $\mathscr{U}$ is exhausted by connected $K$-affinoids i.e., there exists an increasing sequence $\mathscr{U}_1 \subset\mathscr{U}_2\subset \cdots$ of path-connected $K$-affinoid spaces such that $\mathscr{U} = \bigcup \mathscr{U}_i$. 
\item There exists a strictly $K$-affinoid space $M(A)$ and a distinguished closed immersion into some closed polydisk $\iota\colon M(A) \to B^n(0,1)$ such that $\mathscr{U}$ is isomorphic to $\iota(M(A)) \cap B^n(0,1)^-$. 
\end{enumerate}
\end{lemma}

\begin{proof}
These results are proved in \cite[Section 2.5]{Vazquez_NormalFamily}.  We also use \cite[Theorem 3.2.1]{BerkovichSpectral}, which proves that a connected good $K$-analytic space is path-connected. 
\end{proof}

As an immediate consequence of \autoref{defn:boundaryless} and \autoref{lemma:equivalentbasictube}.(2), we have the following. 

\begin{corollary}\label{coro:localbasictubes}
A reduced, boundaryless $K$-analytic space $\sX$ admits a basis of neighborhoods by basic tubes. 
Moreover,  every point $x\in \sX$ admits an path-connected affinoid neighborhood $\sV$ such that $x\in \Int(\sV)$. 
\end{corollary}

\subsection{The kernel and the reduction map of a constant $K$-analytic variety}
\label{subsec:reduction}
Next, we describe the reduction map in the setting of a constant $K$-analytic variety.

First, we recall that the underlying set of $X^{\an}_K$ is
\[
X_K^{\an} = \brk{x = (\xi_x,|\cdot|_x) : \xi_x \in X_K \text{ and } |\cdot|_x\colon \kappa(x)\to \mR}
\]
where $|\cdot|_x\colon \kappa(x)\to \mR$ is an absolute value on the residue field $\kappa(x)$ at $\xi_x$ that extends the absolute value on $K$. 
The field $\sH(x)$ is the completion of $\kappa(x)$ with respect to $|\cdot|_x$.

We also note that there is a \cdef{kernel map} $\iota_{X_K}\colon X_K^{\an} \to X_K$ which sends $x = (\xi_x,|\cdot|_x)$ to $\xi_x$.  

\begin{definition}\label{defn:birational}
The points $x = (\xi_x,|\cdot|_x)$ in $X_K^{\an}$ whose kernel $\iota_{X_K}(x)$ is the generic point of $X_K$ are called \cdef{birational points} of $X_K^{\an}$. 
These points correspond to valuations on the function field of $X_K$ extending the valuation on $K$, and the subspace of birational points on $X_K^{\an}$ is dense in $X_K^{\an}$. 
\end{definition}

Now, we arrive at our description of the reduction map. 
Let $R$ denote the valuation ring of $K$ and consider the base change of $X$ to $R$, which we denote by $X_R$. 
Since $X_R$ is a proper variety over $R$, $X_R$ is a proper $R$-model of $X_K$. 
By using the valuation criterion of properness to the $K$-morphism $\Spec(\sH(x)) \to X_R$, we have a natural reduction map
\[
\red_{X_K^{\an}}\colon X_K^{\an} \to X.
\]
For $x\in X_K^{\an}$, the reduction $\red_{X_K^{\an}}(x)$ is a specialization of the kernel map $\iota_{X_K}(x) \in X_K \subset X_R$. 
The map $\red_{X_K^{\an}}$ is surjective, anti-continuous, and a morphism from the $G$-topological site on $X_K^{\an}$ to the Zariski site of $X$ i.e., for any Zariski affine open $U\subset X$, $\red_{X_K^{\an}}^{-1}(U)$ is an affinoid domain of $X_K^{\an}$. 

Using the fact that $X_K^{\an}$ is the constant $K$-analytic variety associated to a variety (which by definition is irreducible), we may identify a particular point on $X_K^{\an}$ via the reduction map.

\begin{definition}\label{defn:Shilov}
The point $\eta$ in $X_K^{\an}$ corresponding to the unique pre-image of the generic point of $X$ under $\red_{X_K^{\an}}$ is the \cdef{Shilov boundary of $X_K^{\an}$}.  We note that $X_K^{\an}$ deformation retracts onto $\eta$. 
\end{definition}

\subsection{Weakly analytic maps}
\label{subsec:weaklyanalytic}
Recall that for $K$-analytic spaces, pointwise limits of $K$-analytic morphisms need not be $K$-analytic.  Indeed, for $K$-analytic spaces $\sX$ and $\sY$, any constant map $f\colon \sY\to \sX$, $f\equiv x \in \sX$ can be realized as the limit of constant $K$-analytic maps. 
However, the limit of constant $K$-analytic maps need not be $K$-analytic, and in fact, we have that $f$ is analytic if and only if $x \in \sX(K)$.  
This is very different from the complex analytic setting where for $X(\mC)$ and $Y(\mC)$ complex manifolds, the space of holomorphic maps $\text{Hol}(Y(\mC),X(\mC))$ is closed in $\cC(Y(\mC),X(\mC))$.

Despite not being analytic, continuous limits of analytic maps often become analytic after a suitable base change. 
To describe this,  let $\sX$ be a $K$-analytic space and for every complete extension $K'/K$, let $\pi_{K'/K}\colon \sX_{K'}\to \sX$ denote the base change morphism (see \cite[Section 1.4]{BerkovichEtaleCohomology}).  
Clearly, every $K$-point in $\sX$ defines a $K'$-point in $\sX_{K'}$, and since $K$ is algebraically closed, results of Poineau \cite[Section 3]{Poineau:AngelicBerkovich} show that there exists a unique \textit{continuous} extension $\sigma_{K'/K}\colon \sX \hookrightarrow \sX_{K'}$ which defines a section $\pi_{K'/K}$.  We note that the section $\sigma_{K'/K}$ is not a $K$-analytic morphism.

With this construction, we recall Rodr\'iguez V\'azquez's definition of weakly $K$-analytic maps and an important result relating these maps to pointwise limits of $K$-analytic morphisms. 

\begin{definition}[\protect{\cite[Definition 6.1]{Vazquez_NormalFamily}}]
\label{defn:weaklyanalytic}
Let $\sX$ and $\sY$ be $K$-analytic spaces, and let $f\colon \sX\to \sY$ be a continuous map. 
We say that $f$ is \cdef{weakly $K$-analytic} if for every point $x\in X$, there exists an affnoid neighborhood $\sZ$ of $x$, a complete field extension $K'/K$ and a $K'$-analytic map $F\colon \sZ_{K'}\to \sY_{K'}$ such that $f_{|\sZ} = \pi_{K'/K}\circ F \circ \sigma_{K'/K}.$
\end{definition}

\begin{prop}[\protect{\cite[Proposition 6.2]{Vazquez_NormalFamily}}]
\label{prop:locallimitweak}
Let $\sU$ be a basic tube and let $M(A)$ be a $K$-affinoid space. 
Let $f\colon \sU \to M(A)$ be a continuous map. The following are equivalent:
\begin{enumerate}
\item For any point $x\in \sU$, there exists an affinoid neighborhood $\sZ$ of $x$ and a sequence of $K$-analytic maps $f_n\colon \sZ \to M(A)$ pointwise converging to $f_{|\sZ}$.
\item For any point $x\in \sU$, there exists an affinoid neighborhood $\sZ$ of $x$, a complete extension $K'/K$ and a $K'$-analytic map $F\colon \sZ_{K'}\to M(A)_{K'}$ such that $f_{|\sZ} = \pi_{K'/K}\circ F \circ \sigma_{K'/K}$. 
\end{enumerate}
\end{prop}

Although we will not use \autoref{prop:locallimitweak}, we will need an ingredient in Rodr\'iguez V\'azquez's proof.  
Roughly speaking, the proofs of her main results follow from considering the situation where the basic tube $\sU$ is the open $r$-dimensional polydisk $B^r(0,1)^-$ and $M(A)$ is the closed $s$-dimensional polydisk $B^s(0,1)$. 
She constructs an infinite dimensional space $\Mor(B^r(0,1)^-,B^s(0,1))$ that parametrizes analytic maps from $B^r(0,1)^-$ to $B^s(0,1)$.  
We record several facts about this space.

\begin{theorem}[\protect{\cite[Definition 4.3, Proposition 4.4, Proposition 4.6, Theorem 4.9]{Vazquez_NormalFamily}}]\label{thm:Morfacts}
For positive integers $r$ and $s$, we have the following.
\begin{enumerate}
\item $\Mor(B^r(0,1)^-,B^s(0,1))$ is the analytic spectrum of a Banach $K$-algebra $\cT_{\infty}^{r,s}$, and hence is compact. 
\item $\Mor(B^r(0,1)^-,B^s(0,1))$ is homeomorphic an inverse limit of closed unit polydisks. 
\item The set of rigid points in $\Mor(B^r(0,1)^-,B^s(0,1))$ is dense. 
\item $\Mor(B^r(0,1)^-,B^s(0,1))$ is a Fr\'echet--Urysohn space. 
\item Every sequence of rigid points in $\Mor(B^r(0,1)^-,B^s(0,1))$ admits a convergent sub-sequence.  
\end{enumerate}
\end{theorem}

To conclude our discussion on Rodr\'iguez V\'azquez's work, we briefly expand on \autoref{thm:Morfacts}.(2) and describe base change morphisms for $\Mor(B^r(0,1)^-,B^s(0,1))$. 
In \cite[Proposition 4.4]{Vazquez_NormalFamily},Rodr\'iguez V\'azquez constructs a homeomorphism
\[
\Mor(B^r(0,1)^-,B^s(0,1)) \cong \varprojlim_{\delta \in \mN} B^{s(\delta + 1)^r}(0,1)
\]
where $B^{s(\delta + 1)^r}(0,1)$ is the $s(\delta + 1)^r$-dimensional closed unit polydisk over $K$. 
If $K'/K$ is a complete field extension, \cite[p.~1704]{Vazquez_NormalFamily} tells us that we have a base change morphism
\[
\pi_{K'/K}\colon \Mor(B^r(0,1)^-_{K'},B^s(0,1)_{K'}) \cong \Mor(B^r(0,1)^-,B^s(0,1))_{K'} \to \Mor(B^r(0,1)^-,B^s(0,1)). 
\]
Here $\Mor(B^r(0,1)^-_{K'},B^s(0,1)_{K'})$ can be identified with the analytic spectrum of the completed tensor product $\cT_{\infty}^{r,s} \widehat{\otimes}_K K'$. Moreover, the same argument from \cite[Proposition 4.4]{Vazquez_NormalFamily} allows us to homeomorphically identify
\[
\Mor(B^r(0,1)^-_{K'},B^s(0,1)_{K'}) \cong \varprojlim_{\delta \in \mN} B^{s(\delta + 1)^r}(0,1)_{K'}
\]
where $B^{s(\delta + 1)^r}(0,1)_{K'}$ is the $s(\delta + 1)^r$-dimensional closed unit polydisk over $K'$. 
We note that the above inverse limits have the same indexing set and for each $\delta\in \mN$, we have the section map $\sigma_{K'/K}\colon B^{s(\delta + 1)^r}(0,1) \hookrightarrow B^{s(\delta + 1)^r}(0,1)_{K'}$.  By the universal property of inverse limits, we have an injective continuous map 
\begin{equation}\label{eqn:infinitesection}
\sigma_{K'/K,\infty}\colon \Mor(B^r(0,1)^-,B^s(0,1)) \hookrightarrow \Mor(B^r(0,1)^-_{K'},B^s(0,1)_{K'}) \cong \Mor(B^r(0,1)^-,B^s(0,1))_{K'}. 
\end{equation}

\subsection{Non-Archimedean Hilbert schemes}
\label{subsec:NonArchimedeanHilbert}
To conclude our preliminaries on Berkovich spaces, we recall results on Hom-spaces in non-Archimedean geometry.

In \cite[Section 4.1]{conradGAGA}, Conrad defines a rigid analytic variant of the Quot and Hilbert functors, and in \textit{loc.~cit.~}Theorem 4.1.3, he proved that the Hilbert functor for a projective rigid analytic space over $\Sp(K)$ is representable by a separated rigid analytic space over $\Sp(K)$ and that its formation is compatible with analytification. 
Following the scheme theoretic argument, Conrad also defined the rigid analytic Hom-functor and proved (see \textit{loc.~cit.}~Corollary 4.1.5) that the Hom-functor between two projective rigid analytic spaces over $\Sp(K)$ is representable by a separated rigid analytic space over $\Sp(K)$. 

Using comparison results for rigid analytic and Berkovich analytic spaces (see e.g., \cite[Section 8.3]{huber2}),  we have that $\underline{\Hilb}_{X_K^{\an}}$ exists as a strict boundaryless Hausdorff Berkovich $K$-analytic space and $\underline{\Hilb}_{X_K}^{\an} \cong \underline{\Hilb^{\an}}_{X_K^{\an}}$. Similarly, we can define the Berkovich $K$-analytic Hom-space $\underline{\Hom^{\an}}_K(\sX,\sY)$ where $\sX$ and $\sY$ are projective Berkovich $K$-analytic spaces, and for projective varieties $X$ and $Y$ over $K$,  we have that $\underline{\Hom}_K(X,Y)^{\an} \cong \underline{\Hom^{\an}}_K(X^{\an},Y^{\an})$. 

Next, we describe properties concerning the topology of these Berkovich $K$-analytic Hom-spaces.

\begin{lemma}\label{lemma:Kpointscompactopen}
Let $X,Y$ be  projective varieties over $K$. 
The subspace topology on $\underline{\Hom}_K(X,Y)^{\an}(K)$ coincides with compact open topology on the space of $K$-analytic morphisms $X^{\an}\to Y^{\an}$. 
\end{lemma}

\begin{proof}
We note that the construction of the Berkovich analytic Hilbert scheme from \cite[Section 4.1]{conradGAGA} is identical to that of Grothendieck's construction of Hilbert scheme in the complex-analytic category \cite{Grothendieck:Hilb}, and one can show that Grothendieck's Hom-scheme equips set of holomorphic maps with compact-open topology. 
We note that the same strategy works in the Berkovich analytic category, which we sketch below. 
First, we may reduce to the case where $Y = \mP^n_K$ by using the valuative criterion for properness and noting that that a projective embedding on $Y$ induces a closed immersion on Hom-schemes.  
Using the universal property of $\mP^n_K$, one reduces to studying the topology of the Berkovich analytification of the Picard scheme of $X$, and by considering this, one may conclude. 
\end{proof}

\section{\bf Metrizability of connected $K$-affinoids} 
\label{sec:metrizability}

In this section,  we show $K$-affinoid spaces are metrizable when $K $ is an algebraically closed, complete, non-Archimedean non-trivially valued field $K$ of characteristic zero which admits countably dense subset. Furthermore, we describe a metric on a connected $K$-affinoid which defines the $K$-analytic topology and admits minimizing geodesics.


\begin{lemma}\label{lemma:metricaffinoid}
For $K $ an algebraically closed, complete, non-Archimedean non-trivially valued field $K$ of characteristic zero which admits countably dense subset, any $K$-affinoid space $M(A)$ is metrizable. 
\end{lemma}

\begin{proof}
This follows from Urysohn's metrization theorem and the assumptions on $K$. For an explicit construction of a metric, we refer the reader to \cite[Section 34]{Escassut_UltrametricBanach}. 
\end{proof}

\begin{remark}\label{rem:choices}
We note that there is not a unique metric coming from \autoref{lemma:metricaffinoid}. 
Indeed, the metric from \cite[Section 34]{Escassut_UltrametricBanach} depends on a choice of countable dense subset of $K$ and a choice of approximating polynomials for the underlying $K$-Tate algebra. 
\end{remark}

Using basic topological properties of $K$-affinoids, we can find a metric on a $K$-affinoid space which defines the topology and admits minimizing geodesics. 

\begin{lemma}\label{coro:basictubemetrizable}
For $K $ an algebraically closed, complete, non-Archimedean non-trivially valued field $K$ of characteristic zero which admits countably dense subset, any connected $K$-affinoid space $M(A)$ admits a metric $\rho_{M(A)}$ for which $(M(A),\rho_{M(A)})$ is a metric space and any pair of distinct points in $M$ can be joined by a minimizing geodesic with respect to $\rho_{M(A)}$. 
\end{lemma}

\begin{proof}
By \autoref{lemma:metricaffinoid} and the assumptions on $K$, there exists a metric $\rho'$ on $M(A)$ such that $(M(A),\rho')$ is a metric space. 
By \cite[Theorem 1.2.1,  Corollary 2.2.8, and Theorem 3.2.1]{BerkovichSpectral}, the metric space $(M(A),\rho')$ satisfies the conditions of \autoref{thm:affinoidsgeodesic}, and so there exists an inner metric $\rho$ on $M(A)$ for which $(M(A),\rho_{M(A)})$ is a metric space and any pair of distinct points in $M$ can be joined by a minimizing geodesic with respect to $\rho_{M(A)}$.  
\end{proof}


\subsection{Making choices}\label{subsec:makingchoices}
In light of \autoref{rem:choices}, we will make, once and for all, various choices involved in the metrizability of a connected  $K$-affinoid $M(A)$. 
In particular,  we will always assume that $(M(A),\rho)$ is an inner metric space for which any pair of distinct points in $M$ can be joined by a minimizing geodesic with respect to $\rho_{M(A)}$.  
We note that each choice will produce a different metric on $M(A)$,  but the \textit{metrizability} of $M(A)$ is independent of all of these choices.

\section{\bf A non-Archimedean Kobayashi pseudo-metric for constant $K$-analytic varieties}
\label{sec:pseudo}
In this section, we define our pseudo-metric for constant $K$-analytic varieties.

\subsection{Motivation for definition}
Before we begin, we give some motivation for our definition. 
On a connected complex space, the classical Kobayashi pseudo-metric is defined as the infimum of the Poincar\'e metric over all sets of open unit disks connecting distinct points. 
From our perspective, there are four key points in this definition:
\begin{enumerate}
\item the choice of test objects (the open unit disk), 
\item the choice of metric on the test objects (the Poincar\'e metric), 
\item the choice of morphisms from the test objects (holomorphic maps), and 
\item the choice of chains connecting distinct points. 
\end{enumerate}
In this setting, most of these choices are the natural ones; this statement should be thought of as a commentary on the simplicity of the local structure of a complex space.  

In the non-Archimedean setting,  each of the aspects mentioned above requires care due to the complexity of the local structure of a $K$-analytic space.  We refer the reader to \autoref{rem:choiceinvolved} for further discussion on this topic. 

For our test objects, we will use connected $K$-affinoids on constant $K$-analytic varieties. 
Using connected $K$-affinoids  will allow us to define the pseudo-metric on \textit{every} point of $X_K^{\an}$. We pause to remind the reader that $X_K^{\an}$ contains the $K$-points of $X_K$ and additional points corresponding to valuations on residue fields of non-closed points on $X_K$.  
As for the metric on a connected $K$-affinoid, we will use the inner metric described in \autoref{coro:basictubemetrizable}.

Roughly speaking, the types of morphisms we consider from connected $K$-affinoids to $X_K^{\an}$ are $K$-analytic morphisms which are pointed at birational points (\autoref{defn:birational}) and can be extended to a global $K$-analytic map. 
We define and study these morphisms in Subsection \ref{subsec:morhpisms}.

The final input is the choice of chains connecting distinct points in $X_K^{\an}$, and for this construction, we take motivation from work of Demailly, Lempert, and Shiffman \cite{DemaillyLempertShiffman:AlgebraicApproximations}. In this work, the authors showed that one can use algebraic approximation to instead define the classical Kobayashi pseudo-metric on a complex quasi-projective variety via consideration of chains of algebraic curves which are not necessarily smooth or irreducible.

First, we will first define a pseudo-metric on a constant $K$-analytic variety using the above mentioned connected $K$-affinoid and morphisms and the natural definition of chains connecting distinct points in $X_K^{\an}$ (i.e., the one similar to that of the classical Kobayashi pseudo-metric) . 
Then, we will define another pseudo-metric, the $K$-Kobayashi pseudo-metric, on a constant $K$-analytic variety by considering the previously mentioned pseudo-metric and chains of closed integral constant subvarieties of $X_K^{\an}$ connecting two distinct points. 
The reason for doing this is that it allows us to ``see'' all of the subvarieties of $X_K^{\an}$, which allows us to detect when $X_K^{\an}$ admits a morphism from non-hyperbolic varieties.

We will breakdown this section corresponding to three ingredients of our pseudo-metric:~test objects, morphisms, and chains.  

\subsection{Test objects}
As mentioned above, we will use connected $K$-affinoids on constant $K$-analytic varieties as our test objects. We refer the reader to \autoref{defn:constantKanalyticvariety} for the definition of constant $K$-analytic variety. 

\begin{definition}\label{defn:tubesonepsace}
For a constant $K$-analytic variety $X_K^{\an}$, we denote the set of connected $K$-affinoids  on $X_K^{\an}$ by $\mU_{X_K^{\an}}$ i.e., 
\[
\mU_{X_K^{\an}} := \brk{\text{connected $K$-affinoids $M(A)$:~there exists a closed immersion $M(A) \hookrightarrow X_K^{\an}$}}.
\]
\end{definition}

Next, we define a collection of connected $K$-affinoids which will form our test objects for the pseudo-metric. 

\begin{definition}\label{defn:permissible}
Let $X_K^{\an}$ be a constant $K$-analytic variety. 
A \cdef{permissible cover of $X_K^{\an}$} consists of the following data:
\begin{enumerate}
\item a finite collection of connected $K$-affinoids $\sV_1,\dots,\sV_n$ in $\mU_{X_K^{\an}}$ such that $\bigcup_{i=1}^n \Int(\sU_i)= X_K^{\an}$, and 
\item a finite (possibly empty) collection of constant $K$-analytic varieties $Z_{0,K}^{\an},\dots,Z_{\ell,K}^{\an}$ and a finite collection of $K$-affinoids $\brk{\sV_{i,j}}$ in $\mU_{Z_{i,K}^{\an}}$. 
\end{enumerate}
We will denote a permissible cover of $X_K^{\an}$ by 
\[
\underline{\sV} \coloneqq \brk{\underline{\sU}_{X_K^{\an}},\underline{\sU}_{Z_{0,K}^{\an}},\dots,\underline{\sU}_{Z_{\ell,K}^{\an}}}
\] 
where $\underline{\sU}_{\star}$ refers to a finite collection of connected $K$-affinoids on the constant $K$-analytic variety referenced in the subscript, and sometimes we will simply denote it by $\brk{\sU_1,\dots,\sU_m}$ where each of the $\sU_j$ is a connected $K$-affinoid in either $X_K^{\an}$ or $Z_{i,K}^{\an}$ for some $i$. 
\end{definition}

We will define our pseudo-metric using $K$-affinoids from a permissible cover as test objects. 

\begin{remark}\label{remark:basictubecover}
Note that for $X_K^{\an}$ a constant $K$-analytic variety,  one may always find a finite collection of connected $K$-affinoids $\sV_1,\dots,\sV_n$ in $\mU_{X_K^{\an}}$ such that $\bigcup_{i=1}^n \Int(\sU_i)= X_K^{\an}$. 
Indeed,  by \autoref{coro:localbasictubes},  every point $x\in X_K^{\an}$ admits connected $K$-affinoid neighborhood where $x$ is contained in the interior. 
Using these open sets as an open cover of $X_K^{\an}$, compactness of $X_K^{\an}$ tells us that there exists a finite collection of connected $K$-affinoids as described above. 
\end{remark}

\begin{remark}
At first glance, the second piece of data in \autoref{defn:permissible} may seem unnatural. 
The inclusion of additional connected $K$-affinoids on the constant $K$-analytic varieties $Z_{0,K}^{\an},\dots,Z_{\ell,K}^{\an}$ is crucial if one wants to prove a distance decreasing property for the eventual pseudo-metric.  
We note that the distance decreasing property for the classical Kobayashi pseudo-metric is straightforward to see via definitions. We prove a distance decreasing property for our pseudo-metric in \autoref{lemma:distancedecreasing} and the utility of this extra data already appears in Section \ref{sec:vanishing}. 
\end{remark}

\begin{remark}\label{rem:compareCherry}
The choice of test objects signifies the first departure of our pseudo-metric from the Kobayashi pseudo-metric for complex manifolds and from Cherry's extended pseudo-metric on $K$-analytic spaces. 
In \cite{CherryKoba}, Cherry defines a non-Archimedean variant of the Kobayashi pseudo-metric on the rigid points of a $K$-analytic space using the closed unit disk as the test object. 
Note that the closed unit disk is an example of a connected $K$-affinoid, and for a smooth projective $K$-analytic curve, every rigid point admits an open neighborhood isomorphic to the open unit disk. 
Cherry's pseudo-metric is actually an extended pseudo-metric i.e., it may take on the value $\infty$. 
This phenomena occurs when there does not exist a chain of closed unit disks connecting two rigid points. 
As an example, if $\sX$ is a $K$-analytic space whose special fiber $\wt{\sX}$ does not contain any rational curves, then the Cherry extended pseudo-metric between points $x,y \in \sX(K)$ with distinct reductions $\wt{x}\neq\wt{y}$ will be infinite as there is no chain of closed unit disks connecting them (cf.~\cite[Remark p.~133]{CherryKoba}). 
The reason for choosing a larger class of test objects, which depends on the $K$-analytic space $\sX$, is that we need to have test objects which cover $\sX$.  
\end{remark}

\subsection{Morphisms}
\label{subsec:morhpisms}
It is natural to simply consider $K$-analytic morphisms from connected $K$-affinoid to $X_K^{\an}$ for the pseudo-metric. 
However, we want to restrict the class of $K$-analytic morphisms to those which are pointed and roughly speaking can be globally extended to $K$-analytic morphisms. 

\begin{definition}
Let $X_K^{\an}$ and $Y_K^{\an}$ be constant $K$-analytic varieties, and let $\sV \in \mU_{Y_K^{\an}}$ be a connected $K$-affinoid on which contains the point $y\in Y_K^{\an}$. 
\begin{enumerate}
\item Let $x\in X_K^{\an}$. We say that a $K$-analytic morphism $f\colon \sV \to X_K^{\an}$ is \cdef{pointed with respect to $y$ and $x$} if $f(y) = x$.  We will denote this by $f\colon (\sV,y) \to (X_K^{\an},x)$. 
\item A pointed $K$-analytic morphism $f\colon (\sV,y) \to (X_K^{\an},x)$ is \cdef{algebraically extendable} if there exists a pointed morphism of $K$-varieties 
\[
h\colon (\overline{\iota_{Y_K}(\sU)},\iota_{Y_K}(y)) \to (X_K,\iota_{X_K}(x))
\] 
such that ${h^{\an}_{}}_{\vert(\sV,y)} = f$ where $\iota_{Y_K}$ (resp.~$\iota_{X_K}$) is the kernel map of the constant $K$-variety $Y_K^{\an}$ (resp.~$X_K^{\an}$) from Section \ref{subsec:reduction}. 
\end{enumerate}
For $x\in X_K^{\an}$, we will denote the set of algebraically extendable $K$-analytic morphism which are pointed with respect to $y$ and $x$ by $\REMor((\sV,y),(X_K^{\an},x))$. 
\end{definition}

We use pointed algebraically extendable $K$-analytic morphisms in our $K$-Kobayashi pseudo-metric. 

%

\subsection{Chains connecting distinct points}
In this section, we define the notion of a chain of connected $K$-affinoids connecting distinct points on $X_K^{\an}$. 
We emphasize that our pseudo-metric is defined on all of the points of $X_K^{\an}$ and not just the rigid points of $X_K^{\an}$. 

First, we define a labeling of a permissible covering, which is essentially an assignment of points in $X_K^{\an}$ for which we base out pointed algebraically extendable morphisms at. 
We remind the reader of the notion of birational points (\autoref{defn:birational}) and Shilov boundary of $X_K^{\an}$ (\autoref{defn:Shilov}), which will be used below.

\begin{definition}\label{defn:labeling}
Let $X_K^{\an}$ be a constant $K$-analytic variety. 
Let $\{\underline{\sU}_{X_K^{\an}},\underline{\sU}_{Z_{0,K}^{\an}},\dots,\underline{\sU}_{Z_{\ell,K}^{\an}}\}$ be a permissible cover of $X_K^{\an}$. 
Write this permissible cover as $\brk{\sV_1,\dots,\sV_n,\sV_{n+1},\dots,\sV_m}$ where $\sV_1,\dots,\sV_n$ are in $\mU_{X_K^{\an}}$. 

A \cdef{labeling} of $\brk{\sV_1,\dots,\sV_n,\sV_{n+1},\dots,\sV_m}$ is the following data:
\begin{enumerate}
\item for each $i = 1,\dots, m$, a birational point $z_i \in \Int(\sV_i)$ (we note that the birational points are dense in a constant $K$-analytic variety so one always exists),  and
\item a collection of points $x_1,\dots,x_m$
\end{enumerate}
such that for each $i=1,\dots,n$,  $z_i = x_i$ and if the Shilov boundary $\eta$ of $X_K^{\an}$ is contained in $\Int(\sV_i)$, then $z_i = x_i = \eta$. 
We will denote the data of a permissible covering and labeling by 
\[
\brk{(\sV_1,z_1),\dots,(\sV_m,z_m));x_1,\dots,x_m}.
\]
\end{definition}

\begin{definition}\label{defn:pointedcover}
Let $X_K^{\an}$ be a constant $K$-analytic variety,  let $\brk{(\sV_1,z_1),\dots,(\sV_m,z_m));x_1,\dots,x_m}$ be a permissible cover and labeling of $X_K^{\an}$, and let $I = \brk{1,\dots,m}$. 
\begin{enumerate}
\item For any subset $J = \brk{j_1,\dots,j_k}\subseteq I$ where $1\leq k \leq m$, we say that a collection of $K$-analytic morphisms $g_{j_1}\colon \sU_{j_1}\to X_K^{\an},\dots,g_{j_k}\colon \sU_{j_k}\to X_K^{\an}$ is \cdef{pointed with respect to the labeling} if $g_{j_i}(z_{j_i}) = x_{j_i}$ for each $j_i \in J$. 
\item  For any subset $J = \brk{j_1,\dots,j_k}\subseteq I$ where $1\leq k \leq m$, we say that a collection of $K$-analytic morphisms $g_{j_1}\colon \sU_{j_1}\to X_K^{\an},\dots,g_{j_k}\colon \sU_{j_k}\to X_K^{\an}$ \cdef{covers $X_K^{\an}$} if $\bigcup_{i=1}^kg_{j_i}(\sU_{j_i}) = X_K^{\an}$.
\item We say that a permissible cover and labeling $\brk{(\sV_1,z_1),\dots,(\sV_m,z_m));x_1,\dots,x_m}$ of $X_K^{\an}$  \cdef{can cover $X_K^{\an}$} if there exists  a collection of algebraically extendable $K$-analytic morphisms which are pointed with respect to the labeling and covers $X_K^{\an}$. 
\end{enumerate}
\end{definition}

First, we prove that any permissible cover and labeling covers $X_K^{\an}$. 

\begin{lemma}\label{lemma:labelsexist}
Let $X_K^{\an}$ be a constant $K$-analytic variety. 
Any permissible cover and labeling of $X_K^{\an}$ can cover $X_K^{\an}$.  
\end{lemma}

\begin{proof}
Write the permissible cover as $\brk{\sV_1,\dots,\sV_n,\sV_{n+1},\dots,\sV_m}$ where $\sV_1,\dots,\sV_n$ are in $\mU_{X_K^{\an}}$ and form a finite cover of $X_K^{\an}$. 
We recall that the birational points are dense on a constant $K$-analytic variety, and hence for each $i = 1,\dots, m$,  there exists a birational point $z_i \in \Int(\sU_i)$. 
The result follows from the definition of permissible cover and the fact that for each $i = 1,\dots,n$, the natural closed immersion of $(\sV_i,z_i)\hookrightarrow (X_K^{\an},x_i)$ is an algebraically extendable $K$-analytic morphism.  
Indeed, one may take $h$ to be the identity map on $X_K$ in \autoref{defn:permissible}. 
\end{proof}

Next, we define chains connecting two points of a constant $K$-analytic space.

\begin{definition}\label{defn:chain}
Fix a permissible cover and labeling $\brk{(\sV_1,z_1),\dots,(\sV_m,z_m));x_1,\dots,x_m}$ of $X_K^{\an}$. 
Let $I = \brk{1,\dots,m}$, and let $x,y\in X_K^{\an}$ be distinct points. 

For any subset $J = \brk{j_1,\dots,j_k}\subseteq I$ where $1\leq k \leq m$,  let $p_{j_i}, q_{j_i}\in \sU_{j_i}$ be points, and let $f_{j_i} \in \REMor((\sU_{j_i},z_{j_i}),(X_K^{\an},x_{j_i}))$ be algebraically extendable $K$-analytic morphisms which are pointed with respect to the labeling such that
\[x = f_{j_1}(p_{j_1}), f_{j_2}(p_{j_2})= f_{j_1}(q_{j_1}), f_{j_3}(p_{j_3}) = f_{j_2}(q_{j_2}), \ldots, f_{j_k}(p_{j_k}) = f_{j_{k-1}}(q_{j_{k-1}}), f_{j_k}(q_{j_k}) = y.\] 
We define the data 
\[
\alpha({x,y}) \coloneqq \brk{(\sV_{j_1},z_{j_1}),\dots,(\sV_{j_k},z_{j_k}),p_{j_1},\dots,p_{j_k},q_{j_1},\dots,q_{j_k},f_{j_1},\dots,f_{j_k}}
\]
to be a \cdef{chain connecting $x$ to $y$ with links $\brk{(\sV_{j_1},z_{j_1})\dots,(\sV_{j_k},z_{j_k}); x_{j_1},\dots,x_{j_m}}$}. 

The \cdef{length} of a chain connecting $x$ to $y$ is defined to be
\[
\ell(\alpha({x,y})) \coloneqq \sum_{i=1}^k \rho_{{\sV_{j_i}}}(p_{j_i},q_{j_i}). 
\]
where $\rho_{{\sV_{j_i}}}$ is the metric\footnote{We note that we have already made choices involved in this metric in Subsection \ref{subsec:makingchoices}.} from \autoref{coro:basictubemetrizable} on connected $K$-affinoid $\sV_{j_i}$. 
\end{definition}

\begin{remark}
We choose a subset $J = \brk{j_1,\dots,j_k}$ of $I$ for two reasons.  First, we do not want to fix an ordering of the connected $K$-affinoids,  and second, we need to have the possibility of chains which do not contain all of connected $K$-affinoids in the permissible cover $\brk{(\sV_1,z_1)\dots,(\sV_m,z_m)}$. In other words, we allow for the flexibility of excluding some connected $K$-affinoids from some chains.  
\end{remark}

\begin{lemma}
Let $x,y\in X_K^{\an}$ be distinct points. 
For any permissible cover and labeling,  there exists a chain connecting $x$ to $y$. 
\end{lemma}

\begin{proof}
Since $X_K^{\an}$ is path-connected,  this follows from the definition of labeling, \autoref{lemma:labelsexist}, and the chain characterization of connectedness. 
\end{proof}

With this setup, we can construct a pseudo-metric with respect to a fixed permissible cover and labeling. 

\begin{definition}\label{defn:pseudometric}
Fix a permissible cover and labeling $\underline{\sV} = \brk{(\sV_1,z_1)\dots,(\sV_m,z_m);x_1,\dots,x_m}$ of $X_K^{\an}$. 
Let $x,y\in X_K^{\an}$ be distinct points. 
The \cdef{pseudo-metric with respect to $\underline{\sV}$} is defined to be
\[
\rho_{X_K^{\an},\underline{\sV}}(x,y) = \inf_{\alpha({x,y})} \ell(\alpha({x,y}))
\]
where the infimum runs over all chains connecting $x$ to $y$ with links contained in $\underline{\sV}$. 
It is clear that $\rho_{X_K^{\an},\underline{\sV}}$ is symmetric and satisfies the triangle inequality. 
\end{definition}

We now define the non-Archimedean Kobayashi pseudo-metric on the $K$-analytic space $X_K^{\an}$. 

\begin{definition}\label{defn:KKobayashipseudometric}
Let $(Z_{i,K}^{\an})_{i\in \mN}$ denote the countably many constant $K$-analytic subvarieties of $X_K^{\an}$, and let $\underline{\sV}_{Z_{i,K}^{\an}} =\brk{(\sV_{i,1},z_{i,1})\dots,(\sV_{i,m_i},z_{i,m_i}); x_{i,1},\dots,x_{i,m_i}}$ denote a permissible cover and labeling on $Z_{i,K}^{\an}$.  Let $x,y\in X_K^{\an}$ be distinct points.

We define the \cdef{$K$-Kobayashi pseudo-metric with respect to $\underline{\underline{\sV}} = (\underline{\sV}_{Z_{i,K}^{\an}})_{i\in \mN}$} as
\[
d_{X_K^{\an},\underline{\underline{\sV}}}(x,y) = \inf \pwr{\sum_{i=0}^{n-1} \rho_{Z_{i,K}^{\an},\underline{\sV}_{Z_{i,K}^{\an}}} (p_{i},p_{i+1})} 
\]
where the infimum is taken over chains of points $x = p_0,p_1,\dots,p_n = y$ and constant $K$-analytic subvarieties $Z_{0,K}^{\an},\dots,Z_{n,K}^{\an}$ such that 
\[
p_0 \in Z_{0,K}^{\an},p_1 \in Z_{0,K}^{\an}\cap Z_{1,K}^{\an},\dots,p_{n-1} \in Z_{{n-1},K}^{\an}\cap Z_{n,K}^{\an},p_n \in Z_{n,K}^{\an}
\]
and where $\rho_{Z_{i,K}^{\an},\underline{\sV}_{Z_{i,K}^{\an}}}$ is the pseudo-metric with respect to $\underline{\sV}_{Z_{i,K}^{\an}}$. 
Again, it is clear that $d_{X_K^{\an},\uusV}$ is symmetric and satisfies the triangle inequality. 
\end{definition}

\begin{remark}
When $X_K^{\an}$ is a $1$-dimensional constant $K$-analytic variety,  the pseudo-metric $d_{X_K^{\an},\uusV}$ coincides with the pseudo-metric $\rho_{X_K^{\an},\underline{\sV}}$ where $\underline{\sV}$ is a permissible cover and labeling of $X_K^{\an}$. 
\end{remark}

\begin{lemma}\label{lemma:pseudometricontinuous1}
For any permissible cover and labeling $\underline{\sV} = \brk{(\sV_{1},z_{1})\dots,(\sV_{m},z_{m});x_{1},\dots,x_{m}}$ of $X_K^{\an}$, the pseudo-metric with respect to $\underline{\sV}$
\[
\rho_{X_K^{\an},\underline{\sV}}\colon X_K^{\an} \times X_K^{\an} \to \mR_{\geq 0}
\] 
is continuous.
\end{lemma}

\begin{proof}
By the triangle inequality, it suffices to check that if $y_j \to y$ in $X_K^{\an}$, then $\rho_{X_K^{\an},\underline{\sV}}(y_j,y) \to 0$.
Suppose that $\sV_1,\dots,\sV_n$ form a finite cover of $X_K^{\an}$, and without lose of generality, suppose that $y \in \sV_1$.  
Let $f_1\colon (\sV_1,z_1) \hookrightarrow (X_K^{\an},x_1)$ denote the pointed closed immersion of $\sV_1$ into $X_K^{\an}$, and recall that this morphism is algebraically extendable. 
For $j\gg 0$, let $z_j \in \sV_1$ be the point such that $f_1(z_j) = y_j$ where $y_0 = y$.  In particular, $z_j = y_j$. 
Then the sequence $z_j$ tends to $z_0$ and hence as $j\to \infty$, we have that $\rho_{\sV_1}(z_j,z) \to 0$ because $\sV_1$ is metrizable with respect to $\rho_{\sV_1}$ by \autoref{coro:basictubemetrizable}. 
Moreover, we have that $\rho_{X_K^{\an},\underline{\sV}} (y_j,y)$ tends to 0 as $j$ tends to $\infty$, and hence we have our desired result.
\end{proof}

\begin{corollary}\label{lemma:pseudometricontinuous}
For each constant $K$-analytic subvariety $Z_{i,K}^{\an}$ of $X_K^{\an}$, fix a permissible cover and  labeling 
\[
\underline{\sV}_{Z_{i,K}^{\an}} = \brk{((\sV_{i,1},z_{i,1})\dots,(\sV_{i,m_i},z_{i,m_i});x_{i,1},\dots,x_{i,m_i}}
\] 
of $Z_{i,K}^{\an}$. 
Then, the $K$-Kobayashi pseudo-metric with respect to $\uusV = (\underline{\sV}_{Z_{i,K}^{\an}})_{i\in \mN}$
\[
d_{X_K^{\an},\uusV}\colon X_K^{\an} \times X_K^{\an} \to \mR_{\geq 0}
\] 
is continuous.
\end{corollary}

\begin{proof}
This follows immediately from \autoref{lemma:pseudometricontinuous1} since $\rho_{X_K^{\an},\underline{\sU}} \geq d_{X_K^{\an},\uusV}$. 
\end{proof}

\begin{remark}
We note that \autoref{lemma:pseudometricontinuous1} and \autoref{lemma:pseudometricontinuous} tell us that the pseudo-metrics $\rho$ and $d$ are continuous for any choices of permissible cover and labelings. 
\end{remark}

\subsection{A non-Archimedean analogue of pseudo-Kobayashi hyperbolicity}
To conclude this section, we define a notion of hyperbolicity related to the $K$-Kobayashi pseudo-metric. 

\begin{definition}\label{defn:Leahyper}
We say that $X_K^{\an}$ is \cdef{$K$-Lea hyperbolic} (or:~\cdef{$K$-analytically Kobayashi hyperbolic}) if the $K$-Kobayashi pseudo-metric $d_{X_K^{\an},\uusV}$ is a metric with respect to every choice $\uusV$ of permissible cover and labeling on every constant $K$-analytic subvariety. 
\end{definition}

We can easily define the ``pseudoification'' of $K$-Lea hyperbolic. 

\begin{definition}\label{defn:Leahypermodulo}
Let $\Delta \subseteq X_K^{\an}$ be a closed $K$-analytic subspace. 
We say that $X_K^{\an}$ is \cdef{$K$-Lea hyperbolic modulo $\Delta$} (or:~\cdef{$K$-analytically Kobayashi hyperbolic modulo $\Delta$})  if the $K$-Kobayashi pseudo-metric $d_{X_K^{\an},\uusV}$ is a metric modulo $\Delta$ (i.e., $d_{X_K^{\an},\uusV}(x,y) > 0$ unless $x = y$ or $x,y \in \Delta$) with respect to every choice $\uusV$ of permissible cover and labeling on every constant $K$-analytic subvariety. 
\end{definition}

\begin{definition}\label{defn:pseudoLeahyper}
We say that $X_K^{\an}$ is \cdef{pseudo $K$-Lea hyperbolic} (or:~\cdef{pseudo $K$-analytically Kobayashi hyperbolic}) if there exists a proper closed $K$-analytic subspace $\Delta$ of $X_K^{\an}$ such that $X_K^{\an}$ is {$K$-Lea hyperbolic modulo $\Delta$} (or:~{$K$-analytically Kobayashi hyperbolic modulo $\Delta$}).
\end{definition}

\begin{remark}
Clearly, $X_K^{\an}$ is $K$-Lea hyperbolic if and only if $X_K^{\an}$ is $K$-Lea hyperbolic modulo $\emptyset$. 
\end{remark}

\section{\bf Proof of \autoref{xthm:LeaimpliesBarthintro}}
\label{sec:nonArchBarth}
In this section, we prove a non-Archimedean variant of Barth's theorem \cite{Barth:Standard}, which states that for a Kobayashi hyperbolic complex manifold, the Kobayashi metric defines the complex analytic topology. 
We remind the reader of the conventions established in Section \ref{sec:conventions}. 

The main result of this section is the following.

\begin{theorem}[= \autoref{xthm:LeaimpliesBarthintro}]\label{thm:LeaHyperbolicBarth}
Let $X_K^{\an}$ denote a constant $K$-analytic variety. 
If $X_K^{\an}$ is $K$-Lea hyperbolic (\autoref{defn:Leahyper}), then for any choice $\uusV$ of permissible cover and labeling on every constant $K$-analytic subvariety of $X_K^{\an}$, the $K$-Kobayashi metric $d_{X_K^{\an},\uusV}$ defines the Berkovich $K$-analytic topology on $X_K^{\an}$. 
\end{theorem}

Since $X_K^{\an}$ is Hausdorff and locally compact, \autoref{prop:innermetrictopology} tells us that it suffices to show that the $K$-Kobayashi pseudo-metric $d_{X_K^{\an},\uusV}$ is inner (\autoref{defn:innerpseudo}).  

First, we define the tread of a chain. 

\begin{definition}\label{defn:thread}
Let $Z_K^{\an}$ be a constant $K$-analytic subvariety of $X_K^{\an}$, and let 
\[
\underline{\sV} = \brk{(\sV_1,z_1),\dots,(\sV_m,z_m));x_1,\dots,x_m}\] 
be a permissible cover and labeling  of $Z_K^{\an}$. 
Given $x,y\in Z_K^{\an}$, let $\alpha({x,y})$ be a chain connecting $x$ to $y$ with links $\brk{(\sV_{j_1},z_{j_1})\dots,(\sV_{j_k},z_{j_k}); x_{j_1},\dots,x_{j_k}}$. 
For each $\sV_{j_i}$,  \autoref{coro:basictubemetrizable} tells us there exists a minimizing geodesic $\gamma_{j_i}$ between $p_{j_i}$ to $q_{j_i}$ in $\sV_{j_i}$.
Joining $f_{j_1}(\gamma_{j_1}),\dots , f_{j_k}(\gamma_{j_k})$ consecutively, we obtain a path from $x$ to $y$ in $Z_K^{\an}$.

We call this path \cdef{the thread of the chain $\alpha({x,y})$} and denote it by $|\alpha({x,y})|$. 
Clearly, $|\alpha({x,y})|$ is rectifiable and its length $L(|\alpha({x,y})|)$ with respect to $\rho_{Z_K^{\an},\underline{\sV}}$ is bounded by $\ell(\alpha({x,y}))$.
\end{definition}

\begin{lemma}\label{lemma:rhoinner}
Let $X_K^{\an}$ denote a constant $K$-analytic variety.  For any permissible cover and labeling $\underline{\sV} = \brk{(\sU_1,z_1)\dots,(\sU_m,z_m);x_1,\dots,x_m}$ of $X_K^{\an}$,the pseudo-metric $\rho_{X_K^{\an},\underline{\sV}}$ is inner. 
\end{lemma}

\begin{proof}
Let $\rho_{X_K^{\an},\underline{\sV}}^i$ be the inner pseudo-metric induced by $\rho_{X_K^{\an},\underline{\sV}}$. 
Since the inequality $\rho_{X_K^{\an},\underline{\sV}} \leq \rho_{X_K^{\an},\underline{\sV}}^i$ always holds, it suffices to prove $\rho_{X_K^{\an},\underline{\sV}}^i \leq \rho_{X_K^{\an},\underline{\sV}}$. 
For any pair of distinct points $x,y \in X_K^{\an}$, we have 
\[
\rho_{X_K^{\an},\underline{\sV}}^i(x,y) \leq \inf_{\alpha(x,y))} L(|\alpha({x,y})|)\leq \inf_{\alpha({x,y})} \ell(\alpha({x,y})) = \rho_{X_K^{\an},\underline{\sV}}(x,y).
\]
where $\alpha({x,y})$ is a chain connecting $x$ to $y$ and $|\alpha({x,y})|$ is the thread of the chain $\alpha({x,y})$ as defined in \autoref{defn:thread}.   
\end{proof}

\begin{lemma}\label{lemma:Berkseparable}
Any projective constant $K$-analytic variety $X_K^{\an}$ is (locally) compact and separable. 
\end{lemma}

\begin{proof}
Since projectivity is stable under base change, we have that $X_K$ is projective, and so \cite[Theorem 3.4.8]{BerkovichSpectral} tells us that $X_K^{\an}$ is Hausdorff and compact hence locally compact. 
Since $X_K^{\an}$ is compact and metrizable \cite[Theorem 1.1 and Remark 1.5]{HruskovskiLoeserPoonen:BerkovichEmbed}, we have that $X_K^{\an}$ is separable. 
\end{proof}

\begin{lemma}\label{remk:rhodefinesanalytic}
Let $X_K^{\an}$ denote a constant $K$-analytic variety, and let $\underline{\sV} = \brk{(\sU_1,z_1)\dots,(\sU_m,z_m);x_1,\dots,x_m}$ be a permissible cover and labeling of $X_K^{\an}$. If the pseudo-metric $\rho_{X_K^{\an},\underline{\sV}}$ is a metric, then it defines the Berkovich $K$-analytic topology and any pair of distinct points in $X_K^{\an}$ can be joined by a minimizing geodesic with respect to $\rho_{X_K^{\an},\underline{\sV}}$. 
\end{lemma}

\begin{proof}
By \autoref{lemma:pseudometricontinuous1} and \autoref{lemma:rhoinner}, $\rho_{X_K^{\an},\underline{\sV}}$ is a continuous inner pseudo-metric. 
Since $X_K^{\an}$ is locally compact by \autoref{lemma:Berkseparable}, \autoref{prop:innermetrictopology} tells us that $\rho_{X_K^{\an},\underline{\sV}}$ defines the Berkovich $K$-analytic topology. 
Moreover, since $X_K^{\an}$ is also compact (\autoref{lemma:Berkseparable}), the Hopf--Rinow theorem \cite[Theorem 2.5.3]{BuragoBuragoIvanov:CourseMetric} tells us that any pair of distinct points in $X_K^{\an}$ can be joined by a minimizing geodesic with respect to $\rho_{X_K^{\an},\underline{\sV}}$. 
\end{proof}

\begin{prop}\label{prop:innerpseudometric}
Let $X_K^{\an}$ denote a constant $K$-analytic variety. 
If $X_K^{\an}$ is $K$-Lea hyperbolic, then for any choice $\uusV$ of permissible cover and labeling on every closed constant $K$-analytic subvariety of $X_K^{\an}$,  the $K$-Kobayashi metric $d_{X_K^{\an},\uusV}$ is an inner metric. 
\end{prop}

\begin{proof}
Let $d_{X_K^{\an},\uusV}^i$ be the inner pseudo-metric induced by $d_{X_K^{\an}\uusV}$. 
Since the inequality $d_{X_K^{\an},\uusV} \leq d_{X_K^{\an},\uusV}^i$ always holds, it suffices to prove $d_{X_K^{\an},\uusV}^i \leq d_{X_K^{\an},\uusV}$. 
Since $X_K^{\an}$ is $K$-Lea hyperbolic, we know that $d_{X_K^{\an},\uusV}$ is a metric for any choice $\uusV$ of permissible cover and labeling on every constant $K$-analytic subvariety of $X_K^{\an}$. 
In particular,  for any constant $K$-analytic subvariety $Z_K^{\an}$ of $X_K^{\an}$ and any permissible cover and labeling $\underline{\sV}$,the pseudo-metric $\rho_{Z_{K}^{\an},\underline{\sV}}$ is an inner metric (\autoref{lemma:rhoinner}) and any pair of distinct points in $Z_K^{\an}$ can be joined by a minimizing geodesic with respect to $\rho_{Z_K^{\an},\underline{\sV}}$ (\autoref{remk:rhodefinesanalytic}).

Recall that for any pair of distinct points $x,y \in X_K^{\an}$,
\[
d_{X_K^{\an}\uusV}(x,y) = \inf \pwr{\sum_{j=0}^{n-1} \rho_{Z_{j,K}^{\an},\underline{\sV}_j}(p_{j},p_{j+1})} 
\]
where the infimum above is the one from \autoref{defn:KKobayashipseudometric}. 
Using the existence of a minimizing geodesic with respect to $\rho_{Z_K^{\an},\underline{\sV}}$, one can define an analogous notion of a thread (\autoref{defn:thread}) with respect to a chain of constant $K$-analytic subvarieties, and then the same proof as in \autoref{lemma:rhoinner} implies that 
\[
d_{X_K^{\an},\uusV}^i \leq  d_{X_K^{\an},\uusV}.
\]
\end{proof}

\begin{proof}[Proof of \autoref{thm:LeaHyperbolicBarth}]
By \autoref{lemma:pseudometricontinuous}, $d_{X_K^{\an}}$ is continuous, and since $X_K^{\an}$ is locally compact and Hausdorff, the result follows from \autoref{prop:innermetrictopology} and \autoref{prop:innerpseudometric}. 
\end{proof}


\section{\bf Proof of \autoref{xthm:LeaimpliesBrody}}
\label{sec:vanishing}
In this section, we prove a vanishing theorem for the $K$-Kobayashi pseudo-metric on constant $K$-analytic varieties which admit a polarized dynamical system. This result will allow us to prove \autoref{xthm:LeaimpliesBrody}.

First, we recall the definition of a polarized dynamical system and periodic points of a polarized dynamical system.

\begin{definition}
Let $X/L$ be a projective variety. 
A \cdef{polarized dynamical system} on $X$ consists of the following data:
\begin{enumerate}
\item $f\colon X\to X$ an endomorphism of infinite order,
\item an ample line bundle $\mathcal{L}$ on $X$ such that $f^*\mathcal{L} \cong \mathcal{L}^{\otimes q}$ for some $q>1$. 
\end{enumerate}
\end{definition}

\begin{example}\label{exam:polarizeddynamical}
We give two examples of varieties admitting polarized dynamical systems. 
\begin{enumerate}
\item For $n\geq 1$, $\mP^n_L$ admits a polarized dynamical system coming from any degree $d$ morphism $f\colon \mP^n_L\to \mP^n_L$. Indeed, $f^*\sO(1)\cong \sO(d)$. 
\item Any abelian variety $A/L$ admits a polarized dynamical system coming from the multiplication-by-$n$ map $[n]\colon A\to A$ and taking $\mathcal{L}$ to be a symmetric, ample line bundle on $A$. Indeed, in this case, we have that $[n]^*\mathcal{L} \cong \mathcal{L}^{\otimes n^2}$, which can be proved using the theorem of the cube.
\end{enumerate}
\end{example}

We record three facts concerning polarized dynamical systems.

\begin{lemma}\label{lemma:propertiespolarization}
Let $X/L$ be a projective variety.  
\begin{enumerate}
\item If $X$ admits a polarized dynamical system with $f\colon X\to X$ being the endomorphism of infinite order, then $f$ is finite and $\deg(f) = q^{\dim(X)}$. 
\item If $X$ admits a polarized dynamical system, then $\kappa(X)\leq 0$. 
\item Let $L'/L$ be an extension, and let $X_L$ be the base change of $X$ to $L'$.
If $X$ admits a polarized dynamical system, then $X_L$ admits a polarized dynamical system. 
\end{enumerate}
\end{lemma}

\begin{proof}
The first and second facts are well-known (see e.g., \cite[Lemma 1.1.1]{Zhang:DistributionsAlgebraicDynamics} and \cite[Theorem 4.2]{Fakhruddin:Questions}).
The third assertion follows from injectivity of the Picard group under base change \cite[\href{https://stacks.math.columbia.edu/tag/0CDY}{Tag 0CDY}]{stacks-project} and the facts that the base change of an ample line bundle is ample \cite[\href{https://stacks.math.columbia.edu/tag/0893}{Tag 0893}]{stacks-project} and that the base change of a finite moprhism is again finite \cite[\href{https://stacks.math.columbia.edu/tag/01WL}{Tag 01WL}]{stacks-project}. 
\end{proof}

We will be interested in points which are fixed by some power of a polarized dynamical system. 

\begin{definition}
Let $X/L$ be a projective variety admitting a polarized dynamical system $(f,\mathcal{L})$. 
A point $x\in X(L)$ is \cdef{preperiodic for $f$} if $f^{\circ n}(x) = f^{\circ m}(x) $ for some $m>n\geq 0$,  it is called \cdef{periodic for $f$} if we may take $n = 0$. We denote the preperiodic points (resp.~periodic points) of $f$ by $X_{\text{preper}(f)}$ (resp.~$X_{\text{per}(f)}$). 
\end{definition}

The main result of this section is that the $K$-Kobayashi pseudo-metric on a constant $K$-analytic variety associated to a projective variety admitting a polarized dynamical vanishes on the periodic points the pullback of the polarized dynamical system.

\begin{theorem}\label{thm:vanishingdynamical}
Let $X/k$ be a projective variety admitting a polarized dynamical system $(f,\mathcal{L})$,  let $X_K$ be the base change of $X$ to $K$, and let $(f_K,\mathcal{L}_K)$ denote the base change of the polarized dynamical system $(f,\mathcal{L})$ to $X_K$. 
For every choice $\underline{\sV}$ of permissible cover and labeling $X_K^{\an}$, $\rho_{X_K^{\an},\underline{\sV}}(x,y) = 0$ for any pair of periodic points $x,y$ for $(f_K,\mathcal{L}_K)$.  
In particular, $X_K^{\an}$ is not $K$-Lea hyperbolic. 
\end{theorem}

\begin{proof}
Let $\underline{\sV} = \brk{(\sV_1,z_1),\dots,(\sV_m,z_m);x_1,\dots,x_m}$ denote a permissible cover and labeling of $X_K^{\an}$. 
Let $\eta$ denote the Shilov boundary (\autoref{defn:Shilov}) of $X_K^{\an}$ i.e.,  $\eta$ is the unique preimage under the reduction map $\red_{X_K^{\an}}$ of the generic point of $X$. 
Let $\sV_i$ be a connected $K$-affinoid in $\underline{\sV}$ such that $\eta \in \Int(\sV_i)$. 
Note that such a $\sV_i$ appears in every choice of permissible cover and labeling on $X_K^{\an}$. 
By abuse of notation, we will denote this $\sU_i$ by $\sU$. 

By \autoref{lemma:propertiespolarization}.(1,3), we have that $X_K$ admits a polarized dynamical system, namely $(f_K,\mathcal{L}_K)$, and let $\deg(f_K) = q^{\dim(X_K)}$. 
For any $n\geq 1$, consider the morphism $f^{\circ n}_K\colon X_K^{\an}\to X_K^{\an}$, which can be realized as the analytification of base change of the morphism $f^{\circ n}\colon X \to X$ to $K$. 
This follows from GAGA \cite{BerkovichSpectral} and because $\underline{\Hom}_k(X,X) \times_k \Spec(K) \cong \underline{\Hom}_K(X_K,X_K)$. 

Let $y\in X^{\an}_K(K)$ be a rigid point which is periodic with respect to $f_K$. 
Note that there exists infinitely many such $y$ by \cite[Theorem 5.1]{Fakhruddin:Questions}, and in fact, these points are Zariski dense on $X_K$. 
Let $m\geq 1$ be a positive integer such that $y$ is a fixed point for $f_K^{\circ m}$. 
To ease notation, we assume that $m = 1$. 
Since $X_K^{\an}$ is contractible (in particular, $X_K^{\an}$ deformation retracts onto $\eta$), the sequence of probability measures $q^{-\dim(X_K)n}f^{\circ n *}_K\delta_y$ converges to a probability measure $\mu$ whose support is equal to $\eta$ (see e.g., \cite[Theorems 1 and 10]{petsche2009nonarchimedean}), and therefore, we may obtain a sequence of rigid points $y_n $ converging to $\eta$ such that $f^{\circ n}_K(y_n) = y$. 
We have that $f^{\circ n}_{K}(\eta) = \eta$ since $f^{\circ n}_{K}$ is the base change of the $k$-morphism $f^{\circ n}$, and $f^{\circ n}_{K|\sV}\colon \sV \to X_K^{\an}$ is algebraically extendable and pointed with respect to the above labeling.

Moreover, we have that
\[
\rho_{X_K^{\an},\underline{\sV}}(y,\eta) \leq \inf_{\ell\to \infty} \rho_{\sV}(y_\ell,\eta),
\]
where $\rho_{\sV}$ is the metric from \autoref{coro:basictubemetrizable}. 
Note that $y_\ell \in \sV$ for $\ell\gg 0$, and since $\rho_{\sV}$ defines the $K$-analytic topology on $\sV$ and $y_\ell \to \eta$, we have that $\rho_{X_K^{\an},\underline{\sV}}(y,\eta) = 0$. 
To conclude, we note that a similar analysis tells us that for any periodic point $z\in X_K^{\an}(K)$,  we have that $\rho_{X_K^{\an},\underline{\sV}}(\eta,z) = 0$, and thus the triangle inequality tells us that $\rho_{X_K^{\an},\underline{\sV}}$ is zero on the periodic points of $X_K^{\an}$. 
\end{proof}


\begin{corollary}\label{coro:mapfrompolarized}
Let $X$ and $Y$ be projective varieties over $k$. Suppose that $Y$ admits a polarized dynamical system $(f,\mathcal{L})$. 
If there exists a non-constant $k$-morphism $\varphi\colon Y \to X$, then $X_K^{\an}$ is not $K$-Lea hyperbolic. 
\end{corollary}

\begin{proof}
Let $\varphi_K\colon Y_K \to X_K$ denote the base change of the morphism $\varphi$ to $K$. 
Let $(f_K,\mathcal{L}_K)$ denote the base change of the polarized dynamical system $(f,\mathcal{L})$ to $X_K$.
By \cite[Theorem 5.1]{Fakhruddin:Questions}, the periodic points with respect to $f_K$ are Zariski dense on $Y_K$, and hence $\varphi^{\an}_{K\vert Y_{K,\text{per}(f_K)}}$ is non-constant. 
By the proof of \autoref{thm:vanishingdynamical}, we see that the $K$-Kobayashi pseudo-metric will not be a metric for any choice $\underline{\sV}$ of permissible cover of $X_K^{\an}$ which contains a connected $K$-affinoid containing the Shilov boundary of $Y_K^{\an}$,  and hence $X_K^{\an}$ is not $K$-Lea hyperbolic. 
\end{proof}

Using the above result, we can prove \autoref{xthm:LeaimpliesBrody}. 

\begin{prop}\label{prop:leagroupless}
Let $X_K^{\an}$ be a constant $K$-analytic variety. 
If $X_K^{\an}$ is $K$-Lea hyperbolic, then $X$ is groupless. 
\end{prop}

\begin{proof}
Suppose to the contrary that $X$ is not groupless, and so by \autoref{lemma:grouplessequivgeometric}.(2), there exists an abelian variety $A/k$ and a non-constant morphism $f\colon A\to X$. 
Since $A$ admits a polarized dynamical system by \autoref{exam:polarizeddynamical}.(2), \autoref{coro:mapfrompolarized} implies that $X_K^{\an}$ is not $K$-Lea hyperbolic. 
\end{proof}

\begin{theorem}[= \autoref{xthm:LeaimpliesBrody}]
Let $X_K^{\an}$ be a constant $K$-analytic variety. 
If $X_K^{\an}$ is $K$-Lea hyperbolic, then $X_K^{\an}$ is $K$-analytically Brody hyperbolic i.e., for every connected algebraic group $G$ over $K$,  every morphism $G^{\an}\to X_K^{\an}$ is constant. 
\end{theorem}

\begin{proof}
By \cite[Corollary 3.14.]{JVez},  $X_K^{\an}$ being $K$-analytically Brody hyperbolic is equivalent to $X$ being groupless, and so the result follows from \autoref{prop:leagroupless}. 
\end{proof}

We conclude this section with two remarks. 

\begin{remark}
Unfortunately, we cannot prove that $d_{X_K^{\an}}$ is identically zero for any projective variety $X/k$ admitting a polarized dynamical system.  Although the periodic points are Zariski dense by  \cite[Theorem 5.1]{Fakhruddin:Questions}, this does not imply that they are dense in the Berkovich analytic topology as the set of closed points on a projective variety are not locally constructible.  

In fact, one can show that for any rational map $f\colon \mP^{1,\an}_K\to \mP^{1,\an}_K$ of degree $d\geq 2$, the periodic points of $f$ are not dense in the Berkovich analytic topology. 
The closure of the set of periodic points of $f$ in $\mP^{1,\an}_K$ consists of the Julia set $J_f$ and the periodic points (some of which may be in $J_f$). In particular, any non-periodic, non-Julia point is not in the closure of the set of periodic points, and such points (Fatou and non-periodic) necessarily exist. The Fatou set is a nonempty open set, so in particular it contains uncountably many points of type 1. However, the set of periodic points that are type 1 is necessarily countable (since they are the roots of the countably many equations $f^{\circ n}(z)=z$), so some of the type 1 points in the Fatou set must be non-periodic.  We refer the reader to \cite[Chapter 5]{Benedetto:Dynamics} for details on the Julia and Fatou sets. 
\end{remark}

\begin{remark}
We note that the existence of an endomorphism of infinite order is related to the potential density of rational points (see e.g., \cite{AmerikBogomolovRovinsky:RemarksEndomorphisms}). 
\end{remark}

\section{\bf Proof of \autoref{xthm:HyperbolicImpliesLeaIntro}}
\label{sec:proofLeaHyper}
In this section, we prove our main result illustrating how the the hyperbolicity (\autoref{defn:hyperbolic}) of $X$ is related to the $K$-Lea hyperbolicity (\autoref{defn:Leahyper}) of $X_K^{\an}$.  

The main theorem of this section is the following.

\begin{theorem}\label{thm:hyperbolicimpliesLea}
Let $X_K^{\an}$ be a constant $K$-analytic variety. 
If $X$ is hyperbolic, then $X_K^{\an}$ is $K$-Lea hyperbolic. 
\end{theorem}

\autoref{thm:hyperbolicimpliesLea} is in the same spirit as \cite[Theorem 1.3]{JVez} in that one inherits hyperbolicity from the special fiber.  
We will first show that when $X$ is hyperbolic,  the pseudo-metric $\rho_{Z_K^{\an},\underline{\sV}}$ will be a metric for any choice $\underline{\sV}$ of permissible cover and labeling on any constant $K$-analytic subvariety $Z_K^{\an}$ of $X_K^{\an}$.

\begin{prop}\label{thm:hyperbolicimpliesnormal}
Let $X/k$ be a projective hyperbolic variety, and let $X_K^{\an}$ denote the constant $K$-analytic variety associated to $X$.  
For every permissible cover and every labeling 
\[\underline{\sV}_{Z_K^{\an}}= \brk{(\sV_1,z_1),\dots,(\sV_m,z_m);x_1,\dots,x_m}\] 
of every constant $K$-analytic subvariety $Z_K^{\an}$, the family of algebraically extendable $K$-analytic morphisms $\Mor_{\GR}((\sV_i,z_i),(X_K^{\an},x_i))$ is finite for each $1\leq i \leq m$. 
\end{prop}

The proof of \autoref{thm:hyperbolicimpliesnormal} essentially follows from the following lemma.

\begin{lemma}\label{lemma:finitelymanyreduction}
Let $X/k$ be a projective, hyperbolic variety,  let $X_K^{\an}$ denote the constant $K$-analytic variety associated to $X$. 
Let $Z/k$ be a projective variety,  and let $Z_K^{\an}$ denote the constant $K$-analytic variety associated to $Z$. 
Let $z\in Z_K^{\an}$ be a birational point and let $\sV$ be a connected $K$-affinoid neighborhood of $z.$ 
For any point $x\in X_K^{\an}$, the family of algebraically extendable $K$-analytic morphisms $\Mor_{\GR}((\sV,z),(X_K^{\an},x))$ is finite. 
\end{lemma}

\begin{proof}
Since $z \in Z_K^{\an}$ is a birational point (\autoref{defn:birational}), we have that 
\[\overline{\iota_{Z_K}(\sV)} \cong \overline{\iota_{Z_K}(z)} \cong Z_K,\]
where $\iota_{Z_K}$ is the kernel map for $Z_K^{\an}$ (see Subsection \ref{subsec:reduction}) and $\iota_{Z_K}(z)$ is the generic point of $Z_K$.  
This implies that for any $x\in X_K^{\an}$, any algebraically extendable $K$-analytic morphism $f\colon (\sV,z) \to (X_K^{\an},x)$ can be realized as the restriction of the analytification of a morphism of $K$-varieties
\[
h\colon (Z_K,\iota_{Z_K}(z)) \to (X_K,\iota_{X_K}(x))
\]
to $(\sU,z)$. 
By \cite[Corollary 3.4.12]{BerkovichSpectral}, the analytification functor is fully faithful so the $K$-morphism $h$ is unique. 
We will prove that there are only finitely many morphisms $h$. 

Let $X'$ denote the reduced induced scheme structure on the Zariski closure of $\iota_{X_K}(x)$, which we note is an integral projective subvariety of $X_K$. 
Since $\iota_{Z_K}(z)$ is the generic point of $Z_K$,  the morphism $h$ is a dominant morphism between projective $K$-varieties $Z_K$ and $X'$, and hence $h$ is projective. 
Since $h$ is quasi-compact and closed, $h$ is surjective onto its image, \cite[\href{https://stacks.math.columbia.edu/tag/0AH6}{Lemma 0AH6}]{stacks-project}. 
Since $X/k$ is hyperbolic,  \autoref{lemma:basechangehyper} tells us that $X'$ admits a dominant $K$-morphism to a projective $K$-variety $X''$ which is of general type, say $g\colon X' \to X''$. 
Therefore, the composition $g\circ h$ is a dominant $K$-morphism from $Z_K \to X''$,  and so \autoref{thm:KobayashiOchiai} tells us that there are only finitely many such maps. 
This implies that there are only finitely many such maps $h$ and hence finitely many algebraically extendable $K$-analytic morphism from $(\sV,z) \to (X_K^{\an},x)$. 
\end{proof}

\begin{proof}[Proof of \autoref{thm:hyperbolicimpliesnormal}]
This follows immediately from the fact that a subvariety of a hyperbolic variety is also hyperbolic and \autoref{lemma:finitelymanyreduction}. 
\end{proof}

\begin{corollary}\label{coro:hyperbolicpho2}
Let $X/k$ be a projective, hyperbolic variety,  let $X_K^{\an}$ denote the constant $K$-analytic variety associated to $X$. 
For any choice $\underline{\sV}$ permissible cover and labeling on any constant $K$-analytic subvariety $Z_K^{\an}$ of $X_K^{\an}$, the pseudo-metric $\rho_{Z_K^{\an},\underline{\sV}}$ is a metric. 
\end{corollary}

\begin{proof}
First, we note that a subvariety of a hyperbolic variety is also hyperbolic. 
Now the result follows from \autoref{thm:hyperbolicimpliesnormal}. Indeed, since there are only finitely many algebraically extendable $K$-analytic morphisms which are pointed with respect to the labeling, the infimum in the definition of $\rho_{Z_K^{\an},\underline{\sV}}$ is in fact a minimum of non-zero real numbers, and hence $\rho_{Z_K^{\an},\underline{\sV}}$ is non-zero. 
\end{proof}

Next, we will show that if the pseudo-metric $\rho_{Z_K^{\an},\underline{\sV}}$ is a metric for any choice $\underline{\sV}$ of permissible cover and labeling on any constant $K$-analytic subvariety $Z_K^{\an}$ of $X_K^{\an}$, then $X_K^{\an}$ is $K$-Lea hyperbolic. 

\begin{theorem}\label{xprop:normalimpliesLea}
Let $X_K^{\an}$ denote a constant $K$-analytic variety. 
Suppose that for every constant $K$-analytic subvariety $Z_K^{\an}$ of $X_K^{\an}$,  every permissible cover, and every labeling 
\[\underline{\sV}_{Z_K^{\an}}= \brk{(\sV_1,z_1),\dots,(\sV_m,z_m);x_1,\dots,x_m}\] 
of $Z_K^{\an}$,  the pseudo-metric $\rho_{Z_K^{\an},\underline{\sV}_{Z_K^{\an}}}$ is a metric. 
Then,  $X_K^{\an}$ is $K$-Lea hyperbolic. 
\end{theorem}

\begin{proof}
To begin,  we make a choice $\underline{\sV}_{Z_K^{\an}}$ of a permissible cover and labeling for every constant $K$-analytic subvariety $Z_K^{\an}$ of $X_K^{\an}$. 
Let $d_{X_K^{\an},\uusV}$  denote the $K$-Kobayashi pseudo-metric with respect to these choices $\uusV$. 
Suppose that there exist $x,y\in X_K^{\an}$ such that
\[
d_{X_K^{\an}\uusV}(x,y) =  \inf \pwr{\sum_{i=1}^n \rho_{Z_{i,K}^{\an},\underline{\sV}_i} (p_{i-1},p_{i})} = 0
\]
where the infimum is taken over chains of points $x = p_0,p_1,\dots,p_n = y$ and integral constant $K$-analytic subvarieties $Z_{0,K}^{\an},\dots,Z_{n,K}^{\an}$ such that 
\[
p_0 \in Z_{0,K}^{\an},p_1 \in Z_{0,K}^{\an} \cap Z_{1,K}^{\an},\dots,p_{n-1} \in Z_{n-1,K}^{\an}\cap Z_{n,K}^{\an},p_n \in Z_{n,K}^{\an}
\]
and where $\rho_{Z_{i,K}^{\an},\underline{\sV}_i}$ is the pseudo-metric with respect to $\underline{\sV}_i$. 
We will show that $x=  y$.

By our assumptions,  we have that 
$\rho_{Z_K^{\an},\underline{\sV}_{Z_K^{\an}}}$ is a metric. 
This implies that $\rho_{Z_K^{\an},\underline{\sV}_{Z_K^{\an}}}(x,y)$ is non-zero for any constant $K$-analytic subvariety $Z_K^{\an}$ of $X_K^{\an}$, and in particular, every term in the above infimum is non-zero.  
We may extract a sequence of chains of integral closed constant subvarieties such that 
\[
\lim_{j\to \infty}\pwr{\sum_{i_j=1}^{n_j} \rho_{Z_{i_j,K}^{\an},\underline{\sV}_{i_j}} (p_{i_j-1},p_{i_j})} = 0.
\]
Note that $p_0 = x$ and $p_{n_j} = y$ for all $j \in \mN$ and that we may reduce to the setting where the summations in the limit only contain a single constant $K$-analytic subvariety i.e., the setting where
\begin{equation}\label{eqn:single}
\lim_{j\to \infty}\rho_{Z_{j,K}^{\an},\underline{\sV}_{j}} (x,y) = 0.
\end{equation}
After reordering, we may and do assume that $\rho_{Z_{j,K}^{\an},\underline{\sV}_{j}} (x,y) \geq \rho_{Z_{j+1,K}^{\an},\underline{\sV}_{j+1}} (x,y)$.

For each constant $K$-analytic subvariety $Z_{j,K}^{\an}$, recall that $\underline{\sV}_{j}$ denotes a permissible cover and labeling of $Z_{j,K}^{\an}$. 
Let $\underline{\sV}_{X_K^{\an}}$ denote the permissible cover and labeling of $X_K^{\an}$ from the first line of the proof. 
We will build a collection of permissible covers and labelings of $X_K^{\an}$ by piecing together the permissible covers and labeling $\underline{\sV}_{j}$ one at a time. 
Indeed, first consider $j = 0$ i.e., $Z_{0,K}^{\an}$ and $\underline{\sV}_{0}$. 
Then, we can define a permissible cover and labeling $\underline{\sV}_{\underline{0},X_K^{\an}}$ of $X_K^{\an}$ by considering the permissible cover and labeling from $\underline{\sV}$ and then simply adding on the permissible covering and labelings in $\underline{\sV}_{0}$. Note that this will give a permissible cover and labeling of $X_K^{\an}$. 
We can iterate this process i.e., define a sequence of permissible covers and labelings $\underline{\sV}_{\underline{j},X_K^{\an}}$ of $X_K^{\an}$ such that the permissible cover and the points in the labeling of $\underline{\sV}_{\underline{j-1},X_K^{\an}}$ is contained in $\underline{\sV}_{\underline{j},X_K^{\an}}$. 

Our assumption tells us that $\rho_{X_K^{\an},\underline{\sV}_{\underline{j},X_K^{\an}}}$ is a metric for every $j \in \mN$.  
The above convention that $\rho_{Z_{j,K}^{\an},\underline{\sV}_{j}} (x,y) \geq \rho_{Z_{j+1,K}^{\an},\underline{\sV}_{j+1}} (x,y)$ and the definition of $\underline{\sV}_{\underline{j},X_K^{\an}}$ implies that 
\[
\rho_{X_K^{\an},\underline{\sV}_{\underline{j},X_K^{\an}}}(x,y) \geq \rho_{X_K^{\an},\underline{\sV}_{\underline{j+1},X_K^{\an}}}(x,y).
\]
Moreover, we have that 
\[
\rho_{X_K^{\an},\underline{\sV}_{\underline{0},X_K^{\an}}} \geq \rho_{X_K^{\an},\underline{\sV}_{\underline{j},X_K^{\an}}}
\]
for all $j\geq 1$ because we are talking an infimum over a larger collection of permissible covers and labelings.  
Since
\[
\lim_{j\to \infty}\rho_{Z_{j,K}^{\an},\underline{\sV}_{j}} (x,y) = 0,
\]
we have that 
\begin{equation}\label{eqn:finalimit}
\lim_{j\to \infty}\rho_{X_K^{\an},\underline{\sV}_{\underline{j},X_K^{\an}}}(x,y)= 0.
\end{equation}

Recall from \autoref{remk:rhodefinesanalytic} that $(X_K^{\an},\rho_{X_K^{\an},\underline{\sV}_{\underline{0},X_K^{\an}}})$ is a metric space and let $\mathcal{S} = (a_{j,j_m})_{j\in \mN,1\leq j_m \leq n_j}$ denote the sequence of points appearing in \eqref{eqn:single} as the images of points in the links of the chains connecting $x$ to $y$. 
Note that $a_{j,1} = x$ and $a_{j,n_j} = y$ for all $j \in \mN$.
If $n_j = 1$ any $j\in \mN$, then $x = y$. 
If we assume that $n_j \neq 1$ for all $j\in \mN$, we also have that $x = y$. 
Indeed,  \eqref{eqn:finalimit} implies that for each $\varepsilon>0$,  $y$ is in the ball around $x$ of radius $\varepsilon$ with respect to $\rho_{X_K^{\an},\underline{\sV}_{\underline{0},X_K^{\an}}}$.  
Since $(X_K^{\an},\rho_{X_K^{\an},\underline{\sV}_{\underline{0},X_K^{\an}}})$ is a metric space,  this implies that the constant sequence $(y)$ converges to $x$, and since $X_K^{\an} $ is Hausdorff, this implies that $x=y$. 

To conclude, we note that the above argument holds for any choice of permissible cover and labeling on any constant $K$-analytic subvariety, and so  $X_K^{\an}$ is $K$-Lea hyperbolic. 
\end{proof}


\begin{proof}[Proof of \autoref{xthm:HyperbolicImpliesLeaIntro}]
This follows from \autoref{coro:hyperbolicpho2} and \autoref{xprop:normalimpliesLea}. 
\end{proof}

\section{\bf Proof of \autoref{xthm:main1} and corollaries}
\label{sec:proofmain}
In this final section, we prove our main results. 
We remind the reader of the conventions established in Section \ref{sec:conventions}.

First, we will prove a distance decreasing property for the $K$-Kobayashi pseudo-metric.

\begin{lemma}\label{lemma:distancedecreasing}
Let $C/k$ be a smooth projective curve and let $X/k$  projective hyperbolic variety, and let $C_K^{\an}$ (resp.~$X_K^{\an}$) be the constant $K$-analytic variety associated to $C$ (resp.~$X$). 
For any choice of permissible cover and labeling $ \underline{\sV}_{C_K^{\an}} = \brk{(\sV_1,z_1),\dots,(\sV_m,z_m);x_1,\dots,x_m}$ of $C_K^{\an}$,  there exist a choice $\uusV$ of permissible covers and labelings $\underline{\sV}_{Z_K^{\an}}$ of every constant $K$-analytic subvariety $Z_K^{\an}$ of $X_K^{\an}$ such that for any $k$-morphism $f\colon C_K^{\an}\to X_K^{\an}$, we have that for any distinct pair of points $x,y\in C_K^{\an}$, the inequality 
\[
d_{X_K^{\an},\uusV}(f(x),f(y)) \leq d_{C_K^{\an},\underline{\sV}_{C_K^{\an}}}(x,y)
\]
holds. 
\end{lemma}

\begin{proof}
Since $C_K^{\an}$ and $X_K^{\an}$ are proper, any analytic morphism between them is algebraic by Berkovich analytic GAGA \cite[Corollary 3.4.12]{BerkovichSpectral}. 
First, we note that if $f$ is constant, then there is nothing to prove, and so we may assume that $f$ is non-constant. 
If $C_K$ has genus $0$ or $1$, then every map $f\colon C_K\to X_K$ is constant, since $X_K$ is groupless by \autoref{lemma:hyperbolicimpliesgroupless} and \autoref{lemma:grouplessequivgeometric}.(1), and so we may assume that $C_K$ has genus $g\geq 2$. 

Since $C_K$ and $X_K$ are projective, any $k$-morphism $f\colon C_K\to X_K$ is also projective,  in particular closed and quasi-compact, and hence, $f(C_K)$ will surject onto its image.  
The image of $f(C_K)$ must be an irreducible curve inside of $X_K$, and after taking the reduced induced subscheme structure of this image, we may assume that the image is an integral projective curve inside of $X_K$. 
Furthermore, the image of $f$ is a constant curve i.e., a curve of the form $Y_{i,K}$ where $Y_i$ is a curve in $X$. 
By our assumption that $X$ is hyperbolic,  we have that each $Y_{i,K}$ is of general type by \autoref{lemma:genearltypegeometric}. 
For each curve $Y_{i,K}^{\an}$ in $X_K^{\an}$, there are only finitely many non-constant morphisms from $C_K^{\an} \to Y_{i,K}^{\an}$ by \autoref{thm:KobayashiOchiai}.  Let $f_{1,i},\dots,f_{j,i}$ denote the finitely many non-constant morphisms from $C_K^{\an}\to Y_{i,K}^{\an}$. 

Let $\underline{\sU}_{C_K^{\an}} = \brk{(\sU_1,z_1),\dots,(\sU_m,z_m); x_1,\dots,x_m}$ denote a permissible cover and labeling of $C_K^{\an}$. 
We can define a permissible cover of $Y_{i,K}$ as follows:~take any finite collection of connected $K$-affinoids in $\mathbb{U}_{Y_{i,K}^{\an}}$ which form an open cover of $Y_{i,K}^{\an}$ and include the connected $K$-affinoids $(\sU_1,\dots,\sU_m)$ from above. 
We note that this finite collection of connected $K$-affinoids constitutes a permissible cover of $Y_{i,K}$. 
Next, we define a labeling on $Y_{i,K}$ by considering the following collection of points:~any collection of birational points inside of the connected $K$-affinoid which live on $Y_{i,K}^{\an}$ (subject to our Shilov boundary condition) and the finitely many points $\bigcup_{s = 1}^m\bigcup_{t = 1}^j f_{t,i}(x_s)$.  (One may need to add copies of connected $K$-affinoid from $\underline{\sU}_{C_K^{\an}}$ in order to account for all the points in the labeling, but this is fine as there are only finitely many points in $\bigcup_{s = 1}^m\bigcup_{t = 1}^j f_{t,i}(x_s)$.) 
Clearly, this will be a labeling of $Y_{i,K}^{\an}$. 

For all of the other constant $K$-analytic subvarieties $Z_{K}^{\an}$, take any permissible cover and labeling of them, and let $\uusV$ denote the above choices of permissible cover on every constant $K$-analytic subvariety of $X_K^{\an}$. 
Since the definition of the $K$-Kobayashi pseudo-metric involves taking the infimum over all constant $K$-analytic varieties, the choices of permissible covers of $Y_{i,K}^{\an}$ will automatically imply that for any distinct pair of points $x,y\in C_K^{\an}$ and any $k$-morphism $f\colon C_K^{\an} \to X_K^{\an}$, the equality 
\[
d_{X_K^{\an},\uusV}(f(x),f(y)) \leq d_{C_K^{\an},\underline{\sV}_{C_K^{\an}}}(x,y)
\]
will hold. 
\end{proof}

\begin{remark}
We note that \autoref{lemma:distancedecreasing} holds more generally for any constant $K$-analytic variety $Y_K^{\an}$ and for any $k$-analytic morphisms $f\colon Y_K^{\an}\to X_K^{\an}$. 
The only difference in the proof is a keeping track of all of the proper constant $K$-analytic subvarieties of $Y_K^{\an}$. 
In the setting where $Y_K^{\an}$ is a curve,  so there are no proper subvarieties of $Y_K^{\an}$, and since we are interested in studying $k$-analytic morphisms analytifications of constant $K$-analytic curves to $X_K^{\an}$, this setting will suffice for our purposes. 
\end{remark}

\begin{remark}
Let $C_K^{\an}$ be the constant $K$-analytic variety associated to a smooth projective curve $C$, and let $ \underline{\sV}_{C_K^{\an}} $ be a permissible cover and labeling on $C_K^{\an}$. 
In most situations, there does not exist a choice $\underline{\sV}_{X_K^{\an}}$ of permissible cover and labeling on $X_K^{\an}$ such that for any pair of points $x,y\in C_K^{\an}$, we have $\rho_{C_K^{\an},\underline{\sV}_{C_K^{\an}}}(x,y) \geq \rho_{X_K^{\an},\underline{\sV}_{X_K^{\an}}}(f(x),f(y))$ for \textit{every} $k$-analytic morphism $f\colon C_K^{\an} \to X_K^{\an}$. 
This is because the data of a permissible covering and labeling only uses finitely many points of $X_K^{\an}$. Moreover,  if there are infinitely many $k$-morphisms from $f_n\colon C_K^{\an} \to X_K^{\an}$ whose images do not coincide, then one cannot define a permissible cover and labeling on $X_K^{\an}$ that incorporates all of the images of the points in the labeling on $C_K^{\an}$. 
\end{remark}

Next, we will apply the Arzela--Ascoli theorem to show that the space of $k$-analytic morphisms from $C_K^{\an}$ to $X_K^{\an}$ is relatively compact in the space of continuous morphisms from $C_K^{\an}$ to $X_K^{\an}$.

\begin{lemma}\label{lemma:relativelycompact}
Let $C/k$ be a smooth projective curve and let $X/k$  projective hyperbolic variety, and let $C_K^{\an}$ (resp.~$X_K^{\an}$) be the constant $K$-analytic spaces associated to $C$ (resp.~$X$). 
Fix a choice $ \underline{\sV}_{C_K^{\an}} $ of permissible cover and labeling on $C_K^{\an}$. 
The subspace of $k$-analytic maps $\Hom_k^{\an}(C_K^{\an},X_K^{\an})$ in $\mathcal{C}((C_K^{\an},d_{C_K^{\an},\underline{\sV}_{C_K^{\an}} }),(X_K^{\an},d_{X_K^{\an},\uusV}))$ is relatively compact where $\mathcal{C}((C_K^{\an},d_{C_K^{\an},\underline{\sV}_{C_K^{\an}} }),(X_K^{\an},d_{X_K^{\an},\uusV}))$ is endowed with topology of uniform convergence i.e., the topology induced by the metric $e(f,g) :=  \sup_{x\in C_K^{\an}} d_{X_K^{\an},\uusV}(f(x),g(x))$ where $ d_{X_K^{\an},\uusV}$ is the $K$-Kobayashi metric with respect to the choice $\uusV$ of permissible covers and labelings from \autoref{lemma:distancedecreasing}. 
Moreover,  every sequence $(f_n)$ of $k$-analytic maps in $\mathcal{C}((C_K^{\an},d_{C_K^{\an},\underline{\sV}_{C_K^{\an}} }),(X_K^{\an},d_{X_K^{\an},\uusV}))$ admits a subsequence $(f_n')$ which converges uniformly to a continuous map. 
\end{lemma}

\begin{proof}
Let $d_{C_K^{\an},\underline{\sV}_{C_K^{\an}} }$ denote any $K$-Kobayashi pseudo-metric on $C_K^{\an}$. 
By \autoref{lemma:distancedecreasing}, the $K$-Kobayashi pseudo-metric $d_{X_K^{\an},\uusV}$ with respect to $\uusV$ is distance decreasing for any $k$-analytic map $C_K^{\an}\to X_K^{\an}$ i.e., $\Hom_k^{\an}(C_K^{\an},X_K^{\an}) \subset \mathcal{D}((C_K^{\an},d_{C_K^{\an},\underline{\sV}_{C_K^{\an}} }),(X_K^{\an},d_{X_K^{\an},\uusV}))$. 
\autoref{xthm:HyperbolicImpliesLeaIntro} implies that $X_K^{\an}$ is $K$-Lea hyperbolic, and so $d_{X_K^{\an},\uusV}$ is a metric. 
Using \autoref{lemma:Berkseparable},  \autoref{thm:ArzelaAscoli} tells us that 
$\Hom_k^{\an}(C_K^{\an},X_K^{\an})$ is relatively compact in $\mathcal{C}((C_K^{\an},d_{C_K^{\an},\underline{\sV}_{C_K^{\an}} }),(X_K^{\an},d_{X_K^{\an},\uusV}))$. 
The second statement follows from the definition of relatively compact. 
\end{proof}

Now,  we need to translate the relative compactness of 
\[
\Hom_k^{\an}(C_K^{\an},X_K^{\an}) \subset \mathcal{C}((C_K^{\an},d_{C_K^{\an},\underline{\sV}_{C_K^{\an}} }),(X_K^{\an},d_{X_K^{\an},\uusV}))
\] 
to relative compactness in $\underline{\Hom}_K^{\an}(C_K^{\an},X_K^{\an})$. 
In particular, for every sequence $(f_n)$ of $k$-analytic maps,  we will use the second statement of \autoref{lemma:relativelycompact} to construct a complete field extension $K'/K$ and a $K'$-analytic morphism $F\colon C_{K'}^{\an} \to X_{K'}^{\an}$ such that there exists a sub-sequence of $(f_n)$ which converges to $\pi_{K'/K}(F)$ in $\underline{\Hom}_K^{\an}(C_K^{\an},X_K^{\an})$ where $\pi_{K'/K}$ is the base change morphism.

\begin{prop}\label{prop:relativelycompact}
Let $C/k$ be a smooth projective curve and let $X/k$  projective hyperbolic variety, and let $C_K^{\an}$ (resp.~$X_K^{\an}$) be the constant $K$-analytic varieties associated to $C$ (resp.~$X$). 
The subspace of $k$-analytic maps $\Hom_k^{\an}(C_K^{\an},X_K^{\an})$ in $\underline{\Hom}_K^{\an}(C_K^{\an},X_K^{\an})$ is relatively compact. 
\end{prop}

\begin{proof}
First, we note that $\underline{\Hom}_K^{\an}(C_K^{\an},X_K^{\an})$ is a Fr\'echet--Urysohn space by \cite[Th\'eor\`eme 5.3.]{Poineau:AngelicBerkovich}. 
This implies that the closure of $\Hom_k^{\an}(C_K^{\an},X_K^{\an})$ in $\underline{\Hom}_K^{\an}(C_K^{\an},X_K^{\an})$ is its sequential closure. 
Furthermore, this tells us that proving that the subspace of $k$-analytic maps $\Hom_k^{\an}(C_K^{\an},X_K^{\an})$ is relatively compact in $\underline{\Hom}_K^{\an}(C_K^{\an},X_K^{\an})$ amounts to showing that any sequence of $k$-analytic maps admits a converging sub-sequence in $\underline{\Hom}_K^{\an}(C_K^{\an},X_K^{\an})$.

Let $(f_n)$ be a sequence $k$-analytic maps in $\Hom_k^{\an}(C_K^{\an},X_K^{\an})$. 
By considering this as a sequence in $\mathcal{C}((C_K^{\an},d_{C_K^{\an},\underline{\sV}_{C_K^{\an}} }),(X_K^{\an},d_{X_K^{\an},\uusV}))$, \autoref{lemma:relativelycompact} tells us that there exists a sub-sequence $(f_n')$ which uniformly converges to some continuous map $f$.  Let $x$ be any point of $C_K^{\an}$ and consider its image $f(x)\in f(C_K^{\an}) \subset X_K^{\an}$. Pick an affinoid neighborhood $\sZ_{f(x)}'$ of $f(x)$ such that $f(x)$ is contained in the interior $\sU_{f(x)}'$ of $\sZ_{f(x)}'$; note this is possible since $X_K^{\an}$ is boundaryless. 
Since $f$ is continuous, $f^{-1}(\sU_{f(x)}')$ is a boundaryless neighborhood of $x\in C_K^{\an}$, and 
so we can find an affinoid neighborhood $\sZ_x$ of $x$ such that $x$ is contained in the interior $\sU_x$ of $\sZ_x$. 
We may repeat this process for each $x\in C_K^{\an}$, and since $C_K^{\an}$ is compact (\autoref{lemma:Berkseparable}),  we may cover $C_K^{\an}$ by finitely many of theses $\sU_x$,  say $C_K^{\an} = \bigcup_{i=1}^m\sU_{x_i}$.

Consider the restriction of the sequence of $k$-analytic maps $(f_n')$ to $\sZ_{x_1}$. 
By \cite[Theorem 6.4.3/1]{BGR} and \cite[Proposition 2.2]{Vazquez_NormalFamily}, we may extend these restricted morphisms to a sequence of $K$-analytic maps $B^{r_1}(0,1) \to B^{s_1}(0,1)$ for some $r_1$ and $s_1$, and we have that these morphisms agree with $f_{n}'$ when restricted to $\sZ_{x_1}$.  
After restricting these extended maps to $B^{r_1}(0,1)^-$, we obtain a sequence of rigid points (i.e., $K$-analytic morphisms) 
\[
(\widehat{f}_{n}') \in \Mor(B^{r_1}(0,1)^-,B^{s_1}(0,1)),\]
where $\Mor(B^{r_1}(0,1)^-,B^{s_1}(0,1))$ is the space from \autoref{thm:Morfacts}.  
\autoref{thm:Morfacts}.(1,4) tells us that sequence of rigid points $(\widehat{f}_{n}') \in \Mor(B^{r_1}(0,1)^-,B^{s_1}(0,1))$ admits a sub-sequence $(\widehat{f}_{n_1}')$ which will converge to some point $\alpha_1 \in \Mor(B^{r_1}(0,1)^-,B^{s_1}(0,1))$. 
Now consider the restriction of the sub-sequence of $k$-analytic maps $(f_{n_1}')$ to $\sZ_{x_2}$ and repeat the above argument to deduce that sequence of rigid points $(\widehat{f}_{n_1}') \in \Mor(B^{r_2}(0,1)^-,B^{s_2}(0,1))$ admits a sub-sequence $(\widehat{f}_{n_2}')$ which will converge to some point $\alpha_2 \in \Mor(B^{r_2}(0,1)^-,B^{s_2}(0,1))$. 
By doing this procedure $m$-many times, we find a sub-sequence $(f_{n}'')$ of $(f_n')$ which converges uniformly to $f$ such that for each $i=1,\dots,m$, when we apply the above procedure to consider $(f_{n}'')$ as a sequence of $K$-analytic maps $(\widehat{f}_{n}'') $ in $ \Mor(B^{r_i}(0,1)^-,B^{s_i}(0,1))$, these morphisms agree with $(f_{n}'')$ when restricted to $\sZ_{x_i}$ and $(\widehat{f}_{n}'')$ converges to $\alpha_i \in \Mor(B^{r_i}(0,1)^-,B^{s_i}(0,1))$.

Let $K'$ denote the compositum of the fields $\sH(\alpha_i)$ for $i=1,\dots,m$ and recall the construction of the continuous map \[
\sigma_{K'/K,\infty,i}\colon \Mor(B^{r_i}(0,1)^-,B^{s_i}(0,1)) \hookrightarrow \Mor(B^{r_i}(0,1)^-,B^{s_i}(0,1))_{K'} \cong \Mor(B^{r_i}(0,1)^-_{K'},B^{s_i}(0,1)_{K'})
\] 
from \eqref{eqn:infinitesection}. 
For each $i=1,\dots,m$, consider $\sigma_{K'/K,\infty,i}(\alpha_i) \in \Mor(B^{r_i}(0,1)^-_{K'},B^{s_i}(0,1)_{K'})$. 
By definition of $K'$, the image $\sigma_{K'/K,\infty,i}(\alpha_i)$ is a rigid point in $\Mor(B^{r_i}(0,1)^-_{K'},B^{s_i}(0,1)_{K'})$, and hence corresponds to a $K'$-analytic morphism (cf.~\cite[Theorem 4.8.(i)]{Vazquez_NormalFamily}). 
Since $C_K^{\an} = \bigcup_{i=1}^m\sU_{x_i}$, we have that $C_{K'}^{\an}$ is covered by $\bigcup_{i=1}^m \sU_{x_i,K'}$ and hence by $\bigcup_{i=1}^m \sZ_{x_i,K'}$. 
For each $i=1,\dots,m$, restrict the $K'$-analytic morphism $\sigma_{K'/K,\infty,i}(\alpha_i)$ to $\sZ_{x_i,K'}$. 
We claim that these restricted morphisms glue together to a $K'$-analytic morphism $F\colon C_{K'}^{\an} \to X_{K'}^{\an}$.  
First, we may glue them to construct a continuous map. 
It suffices to check that the the $K'$-analytic morphisms $\sigma_{K'/K,\infty,i}(\alpha_i)$ agree on overlaps. 
For each $i=1,\dots,m$, the convergence in $\Mor(B^{r_i}(0,1)^-_{K'},B^{s_i}(0,1)_{K'})(K')$, the subspace of rigid points on $\Mor(B^{r_i}(0,1)^-_{K'},B^{s_i}(0,1)_{K'})$, is compact convergence since $B^{s_i}(0,1)_{K'} $ is compact and Hausdorff. 
To check that we have agreement on overlaps, note that for any pair $i,j$, $\sZ_{x_i,K’} \cap \sZ_{x_j,K’} = \sZ_{ij,K’}$ where $\sZ_{ij}$ is some affinoid (hence compact) subspace of $C_{K}^{\an}$ which follows from separatedness of $C_{K}^{\an}$. 
Since $(\sigma_{K’/K,\infty,i}(\widehat{f}_{n}'')_{| \sZ_{ij,K’}})$ is equal to $(\sigma_{K’/K,\infty,j}(\widehat{f}_{n}'')_{| \sZ_{ij,K’}})$, their respective limits $\sigma_{K’/K,\infty,i}(\alpha_i)_{| \sZ_{ij,K’}}$ and  $\sigma_{K’/K,\infty,j}(\alpha_j)_{| \sZ_{ij,K’}}$ are equal, and so we can construct a globally defined continuous map $F\colon C_{K'}^{\an} \to X_{K'}^{\an}$. 
Moreover, we can define an atlas on $C_{K'}^{\an}$ and on $X_{K'}^{\an}$ using the $\sU_{x_i,K'}$, and with this, we have that $F$ is in fact a $K'$-analytic morphism as in \autoref{defn:analyticmorphism}. 
%

To conclude our proof, consider the base change morphism
\[
\pi_{K'/K}\colon  \underline{\Hom}_{K'}^{\an}(C_{K'}^{\an},X_{K'}^{\an})\cong  \underline{\Hom}_K^{\an}(C_K^{\an},X_K^{\an})_{K'} \to  \underline{\Hom}_K^{\an}(C_K^{\an},X_K^{\an})
\]
and its section
\[
\sigma_{K'/K}\colon \underline{\Hom}_K^{\an}(C_K^{\an},X_K^{\an}) \hookrightarrow  \underline{\Hom}_{K'}^{\an}(C_{K'}^{\an},X_{K'}^{\an}).
\]
We claim that $(\sigma_{K'/K}(f_{n}''))$ converges to $F$ in $\underline{\Hom}_{K'}^{\an}(C_{K'}^{\an},X_{K'}^{\an})$. 
Note that $(\sigma_{K'/K}(f_{n}''))$ and $F$ are $K'$-analytic morphisms i.e., they lie in $\underline{\Hom}_{K'}^{\an}(C_{K'}^{\an},X_{K'}^{\an})(K')$. 
By \autoref{lemma:Kpointscompactopen}, the subspace topology on $\underline{\Hom}_{K'}^{\an}(C_{K'}^{\an},X_{K'}^{\an})(K')$ coincides with the compact-open topology, and hence it suffices to show that $(\sigma_{K'/K}(f_{n}''))$ converges to $F$ in the compact-open topology. 
This convergence follows directly from the arguments in \cite[Proof of Theorem 4.8.(i)]{Vazquez_NormalFamily}.  
More precisely,  in \textit{loc.~cit.~}notation, the map $\text{Ev}_{K'}(\cdot)$ is continuous on rigid points of $ \Mor(B^r(0,1)^-_{K'},B^s(0,1)_{K'})$ and then convergence of $(\sigma_{K'/K}(f_{n}''))$ to $F$ follows from construction. 
Since $\pi_{K'/K}$ is continuous and $\sigma_{K'/K}$ is a section of $\pi_{K'/K}$, we have that 
\[
\pi_{K'/K}((\sigma_{K'/K}(f_{n}''))) = (f_{n}'') \to \pi_{K'/K}(F)
\]
in $\underline{\Hom}_K^{\an}(C_K^{\an},X_K^{\an})$. 
Therefore, we have shown that sequence of $k$-analytic maps admits a converging sub-sequence i.e.,  $\Hom_k(C_K^{\an},X_K^{\an})$ is relatively compact in $\underline{\Hom}_K^{\an}(C_K^{\an},X_K^{\an})$.  
\end{proof}

\begin{proof}[Proof of \autoref{xthm:main1}]
Let $L$ be an algebraically closed field of characteristic zero and $X/L$ be a  projective hyperbolic variety. 
By \autoref{thm:equivalentbounded}, it suffices to show $X$ is $1$-bounded over $L$. 

To begin,  we may and do assume that $X$ is defined over a countable, algebraically closed field $k$.  
Indeed, the assumption of finite type implies that $X$ is determined by a finite amount of data and each piece of data is defined over a finitely generated field extension of the prime subfield of $L$. Since $L$ has characteristic zero, the prime subfield of $L$ is isomorphic to $\mQ$, and hence countable. Furthermore,  $X$ can be defined over an algebraic closure $k$ of a finitely generated extension of $\mQ$, which is countable.  
Let $K$ denote the completion of an algebraic closure of the formal Laurent series with coefficients in $k$ with respect to the $t$-adic valuation.

Using \autoref{prop:relativelycompact} and Berkovich analytic GAGA \cite[Corollary 3.4.12]{BerkovichSpectral}, we have that $\Hom_k(C_K,X_K)$ is relatively compact in $\underline{\Hom}_K(C_K,X_K)^{\an}$. 
Consider the base change morphism
\[
\pi_{K/k}\colon \underline{\Hom}_K(C_K,X_K)^{\an}\cong \underline{\Hom}_k(C,X)^{\an}_K \to \underline{\Hom}_k(C,X)^{\an}
\]
where we are using the Berkovich analytification of schemes over the trivially valued field $k$ to define the codomain of $\pi_{K/k}$. 
Note that $ \underline{\Hom}_K(C_K,X_K)^{\an}$ (resp.~$\underline{\Hom}_k(C,X)^{\an}$) is Hausdorff by \cite[Theorem 3.4.8]{BerkovichSpectral} (resp.~\cite[Theorem 3.5.3]{BerkovichSpectral}) and since $\pi_{K/k}$ is continuous,  the image $\pi_{K/k}(\Hom_k(C_K,X_K)) \cong \Hom_k(C,X)$ is relatively compact in $\underline{\Hom}_k(C,X)^{\an}$. 
Indeed, we have that $\overline{\pi_{K/k}(\Hom_k(C_K,X_K)) } \cong \overline{\pi_{K/k}(\overline{\Hom_k(C_K,X_K)) }}$, and the latter is compact since the continuous image of compact is compact and the closure of a compact subset in a Hausdorff space is compact. 

By \cite[Theorem 3.5.1.(i)]{BerkovichSpectral}, the closure of $\Hom_k(C,X)$ in $\underline{\Hom}_k(C,X)^{\an}$ is the set of points $ x\in \underline{\Hom}_k(C,X)^{\an}$ for which the valuation on $\sH(x)$ is trivial, and we denote this closure by $(\underline{\Hom}_k(C,X)^{\an})^{\triv}$. 
We have that $(\underline{\Hom}_k(C,X)^{\an})^{\triv}$ is compact because $\Hom_k(C,X)$ is relatively compact in $\underline{\Hom}_k(C,X)^{\an}$. 
Moreover,  \cite[Theorem 3.5.1.(i)]{BerkovichSpectral} tells us that 
\[
\iota_{ \underline{\Hom}_k(C,X)}\colon \underline{\Hom}_k(C,X)^{\an} \to \underline{\Hom}_k(C,X)
\] admits a section over $(\underline{\Hom}_k(C,X)^{\an})^{\triv}$. 
Therefore, we have a continuous bijection\footnote{Note that this is not a homeomorphism as $\underline{\Hom}_k(C,X)$ is not Hausdorff in the Zariski topology.} 
\[
(\underline{\Hom}_k(C,X)^{\an})^{\triv} \to  \underline{\Hom}_k(C,X).
\]
Since $(\underline{\Hom}_k(C,X)^{\an})^{\triv} $ is compact, we have that $\underline{\Hom}_k(C,X)$ is quasi-compact 
and hence of finite type. 
Therefore, $X$ is $1$-bounded over $k$ and by \autoref{thm:boundedgeometric}, $X$ is also $1$-bounded over $L$. 
\end{proof}

\begin{proof}[Proof of \autoref{xcoro:main1}]
\autoref{xthm:main1} shows that (1)$\Rightarrow$(2). 
By \cite[Proposition 4.4]{JKam}, we know that (2)$\Rightarrow$(3), and finally \autoref{lemma:grouplessequivgeometric} and \cite[Remark 3.4 \& Lemma 3.23]{JavanpeykarXie:Finiteness} imply that (3)$\Rightarrow$(1). 
\end{proof}

\begin{proof}[Proof of \autoref{xcoro:main2}]
By \autoref{xthm:main1}, we have that a  projective hyperbolic variety $X$ over $L$ is $1$-bounded over $L$, and \autoref{thm:equivalentbounded} asserts that $1$-boundedness is equivalent to boundedness (\autoref{defn:bounded}). 
Recall that $X$ does not admit any rational curves by \autoref{lemma:grouplessequivgeometric} and hence \cite[Proposition 3.9]{JKam} allows us to deduce that for any normal projective variety $Y/L$, the Hom-scheme $\underline{\Hom}_L(Y,X)$ is a finite disjoint union of projective schemes and moreover is projective.
\end{proof}

\begin{proof}[Proof of \autoref{xcoro:main3}]
This follows from \autoref{xthm:main1} and \cite[Theorem 1.5]{Javanpeykar:FiniteMapsSurfaces}. 
\end{proof}

\begin{proof}[Proof of \autoref{xcoro:main4}]
This follows from \autoref{xthm:main1} and \cite[Remark 4.3 \& Lemma 4.6]{JKam}. 
\end{proof}

\begin{proof}[Proof of \autoref{xcoro:main5}]
This follows from \autoref{xcoro:main4} and \cite[Theorem 1.6]{Javanpyekar:ArithmeticHyperbolicity}. 
\end{proof}

  \bibliography{refs}{}

\def\cprime{$'$}
\providecommand{\bysame}{\leavevmode\hbox to3em{\hrulefill}\thinspace}
\providecommand{\MR}{\relax\ifhmode\unskip\space\fi MR }
\providecommand{\MRhref}[2]{%
  \href{http://www.ams.org/mathscinet-getitem?mr=#1}{#2}
}
\providecommand{\href}[2]{#2}
\begin{thebibliography}{KRZB16}

\bibitem[ABR11]{AmerikBogomolovRovinsky:RemarksEndomorphisms}
E.~Amerik, F.~Bogomolov, and M.~Rovinsky, \emph{Remarks on endomorphisms and
  rational points}, Compos. Math. \textbf{147} (2011), no.~6, 1819--1842.
  \MR{2862064}

\bibitem[ACW08]{AnCherryWang}
Ta~Thi~Hoai An, W.~Cherry, and Julie Tzu-Yueh Wang, \emph{Algebraic degeneracy
  of non-{A}rchimedean analytic maps}, Indag. Math. (N.S.) \textbf{19} (2008),
  no.~3, 481--492. \MR{2513064}

\bibitem[Bar72]{Barth:Standard}
Theodore~J. Barth, \emph{The {K}obayashi distance induces the standard
  topology}, Proc. Amer. Math. Soc. \textbf{35} (1972), 439--441. \MR{306545}

\bibitem[BBI01]{BuragoBuragoIvanov:CourseMetric}
Dmitri Burago, Yuri Burago, and Sergei Ivanov, \emph{A course in metric
  geometry}, Graduate Studies in Mathematics, vol.~33, American Mathematical
  Society, Providence, RI, 2001. \MR{1835418}

\bibitem[BD18]{BrotbekDarondeau:CompleteIntersectionsAmple}
Damian Brotbek and Lionel Darondeau, \emph{Complete intersection varieties with
  ample cotangent bundles}, Invent. Math. \textbf{212} (2018), no.~3, 913--940.
  \MR{3802300}

\bibitem[Ben19]{Benedetto:Dynamics}
Robert~L. Benedetto, \emph{Dynamics in one non-archimedean variable}, Graduate
  Studies in Mathematics, vol. 198, American Mathematical Society, Providence,
  RI, 2019. \MR{3890051}

\bibitem[Ber90]{BerkovichSpectral}
Vladimir~G. Berkovich, \emph{Spectral theory and analytic geometry over
  non-{A}rchimedean fields}, Mathematical Surveys and Monographs, vol.~33,
  American Mathematical Society, Providence, RI, 1990. \MR{1070709}

\bibitem[Ber93]{BerkovichEtaleCohomology}
\bysame, \emph{\'{E}tale cohomology for non-{A}rchimedean analytic spaces},
  Inst. Hautes \'{E}tudes Sci. Publ. Math. (1993), no.~78, 5--161 (1994).
  \MR{1259429}

\bibitem[BGR84]{BGR}
S.~Bosch, U.~G\"untzer, and R.~Remmert, \emph{Non-{A}rchimedean analysis},
  Grundlehren der Mathematischen Wissenschaften [Fundamental Principles of
  Mathematical Sciences], vol. 261, Springer-Verlag, Berlin, 1984, A systematic
  approach to rigid analytic geometry. \MR{746961}

\bibitem[Bin49]{Bing:partitioning}
R.~H. Bing, \emph{Partitioning a set}, Bull. Amer. Math. Soc. \textbf{55}
  (1949), 1101--1110. \MR{35429}

\bibitem[Blo26]{Bloch26}
Andr\'{e} Bloch, \emph{Sur les syst\`{e}mes de fonctions uniformes satisfaisant
  \`{e} l'\'{e}quation d'une vari\'{e}t\'{e} alg\'{e}brique dont
  l'irr\'{e}gularit\'{e} d\'{e}passe la dimension}, J. Math. Pures Appl.
  \textbf{5} (1926), 9--66.

\bibitem[Bom90]{Bombieri:Mordell}
Enrico Bombieri, \emph{The {M}ordell conjecture revisited}, Ann. Scuola Norm.
  Sup. Pisa Cl. Sci. (4) \textbf{17} (1990), no.~4, 615--640. \MR{1093712}

\bibitem[Bro11]{Brobek:Thesis}
Damian Brotbek, \emph{Projective varieties with ample cotangent bundle}, Ph.D.
  thesis, Universit{\'e} de {R}ennes {$1$}, 2011.

\bibitem[Bru22]{Brunebarbe:HyperbolicityLarge}
Yohan Brunebarbe, \emph{Hyperbolicity in presence of a large local system},
  Preprint, arXiv:2207.03283 (July 7, 2022).

\bibitem[Che94]{Cherry}
William Cherry, \emph{Non-{A}rchimedean analytic curves in abelian varieties},
  Math. Ann. \textbf{300} (1994), no.~3, 393--404. \MR{1304429}

\bibitem[Che96]{CherryKoba}
\bysame, \emph{A non-{A}rchimedean analogue of the {K}obayashi semi-distance
  and its non-degeneracy on abelian varieties}, Illinois J. Math. \textbf{40}
  (1996), no.~1, 123--140. \MR{1386317}

\bibitem[Con06]{conradGAGA}
Brian Conrad, \emph{Relative ampleness in rigid geometry}, Ann. Inst. Fourier
  (Grenoble) \textbf{56} (2006), no.~4, 1049--1126. \MR{2266885}

\bibitem[CR19]{CoskunRiedl:AlgHyperGeneralQuintic}
Izzet Coskun and Eric Riedl, \emph{Algebraic hyperbolicity of the very general
  quintic surface in {$\mathbb{P}^3$}}, Adv. Math. \textbf{350} (2019),
  1314--1323. \MR{3949983}

\bibitem[Dem97]{Demailly}
J.-P. Demailly, \emph{Algebraic criteria for kobayashi hyperbolic projective
  varieties}, Proc. Symp. Pure Math. \textbf{62.2} (1997), 285--360.

\bibitem[DGH21]{DGH:Uniform}
Vesselin Dimitrov, Ziyang Gao, and Philipp Habegger, \emph{Uniformity in
  {M}ordell-{L}ang for curves}, Ann. of Math. (2) \textbf{194} (2021), no.~1,
  237--298. \MR{4276287}

\bibitem[DLS94]{DemaillyLempertShiffman:AlgebraicApproximations}
Jean-Pierre Demailly, L\'{a}szl\'{o} Lempert, and Bernard Shiffman,
  \emph{Algebraic approximations of holomorphic maps from {S}tein domains to
  projective manifolds}, Duke Math. J. \textbf{76} (1994), no.~2, 333--363.
  \MR{1302317}

\bibitem[Ein88]{Ein:SubvaritiesGeneral}
Lawrence Ein, \emph{Subvarieties of generic complete intersections}, Invent.
  Math. \textbf{94} (1988), no.~1, 163--169. \MR{958594}

\bibitem[Esc03]{Escassut_UltrametricBanach}
Alain Escassut, \emph{Ultrametric {B}anach algebras}, World Scientific
  Publishing Co., Inc., River Edge, NJ, 2003. \MR{1978961}

\bibitem[Fak03]{Fakhruddin:Questions}
Najmuddin Fakhruddin, \emph{Questions on self maps of algebraic varieties}, J.
  Ramanujan Math. Soc. \textbf{18} (2003), no.~2, 109--122. \MR{1995861}

\bibitem[Fal83]{Faltings2}
Gerd Faltings, \emph{Endlichkeitss\"atze f\"ur abelsche {V}ariet\"aten \"uber
  {Z}ahlk\"orpern}, Invent. Math. \textbf{73} (1983), no.~3, 349--366.
  \MR{718935 (85g:11026a)}

\bibitem[Fal91]{FaltingsLang1}
\bysame, \emph{Diophantine approximation on abelian varieties}, Ann. of Math.
  (2) \textbf{133} (1991), no.~3, 549--576. \MR{1109353}

\bibitem[Fal94]{FaltingsLang2}
\bysame, \emph{The general case of {S}. {L}ang's conjecture}, Barsotti
  {S}ymposium in {A}lgebraic {G}eometry ({A}bano {T}erme, 1991), Perspect.
  Math., vol.~15, Academic Press, San Diego, CA, 1994, pp.~175--182.
  \MR{1307396}

\bibitem[Fav15]{Favre:Countable}
Charles Favre, \emph{Countability properties of some {B}erkovich spaces},
  Berkovich spaces and applications, Lecture Notes in Math., vol. 2119,
  Springer, Cham, 2015, pp.~119--132. \MR{3330764}

\bibitem[GG80]{GreenGriffiths:Conj}
Mark Green and Phillip Griffiths, \emph{Two applications of algebraic geometry
  to entire holomorphic mappings}, The {C}hern {S}ymposium 1979 ({P}roc.
  {I}nternat. {S}ympos., {B}erkeley, {C}alif., 1979), Springer, New
  York-Berlin, 1980, pp.~41--74. \MR{609557}

\bibitem[Gro61]{Grothendieck:Hilb}
Alexander Grothendieck, \emph{Techniques de construction en g{\'e}om{\'e}trie
  analytique. {IX}. {Q}uelques probl{\`e}mes de modules}, S{\'e}minaire Henri
  Cartan \textbf{13} (1960-1961), no.~2, 1-- 20 (fre).

\bibitem[Gro95]{GrothendieckHilbertSchemes}
\bysame, \emph{Techniques de construction et th\'eor\`emes d'existence en
  g\'eom\'etrie alg\'ebrique. {IV} ({FGA}). {L}es sch\'emas de {H}ilbert},
  S\'eminaire {B}ourbaki, {V}ol.\ 6, Soc. Math. France, Paris, 1995, pp.~Exp.\
  No.\ 221, 249--276. \MR{1611822}

\bibitem[HLP14]{HruskovskiLoeserPoonen:BerkovichEmbed}
Ehud Hrushovski, Fran\c{c}ois Loeser, and Bjorn Poonen, \emph{Berkovich spaces
  embed in {E}uclidean spaces}, Enseign. Math. \textbf{60} (2014), no.~3-4,
  273--292. \MR{3342647}

\bibitem[HM06]{HaconMcKernan:BoundednessPluricanonical}
Christopher~D. Hacon and James McKernan, \emph{Boundedness of pluricanonical
  maps of varieties of general type}, Invent. Math. \textbf{166} (2006), no.~1,
  1--25. \MR{2242631}

\bibitem[Hub94]{huber2}
Roland Huber, \emph{A generalization of formal schemes and rigid analytic
  varieties}, Math. Z. \textbf{217} (1994), no.~4, 513--551.

\bibitem[Jav20]{Javanpeykar:Survey}
Ariyan Javanpeykar, \emph{The {L}ang-{V}ojta conjectures on projective
  pseudo-hyperbolic varieties}, Arithmetic geometry of logarithmic pairs and
  hyperbolicity of moduli spaces---hyperbolicity in {M}ontr\'{e}al, CRM Short
  Courses, Springer, Cham, [2020] \copyright 2020, pp.~135--196. \MR{4294877}

\bibitem[Jav21a]{Javanpyekar:ArithmeticHyperbolicity}
\bysame, \emph{Arithmetic hyperbolicity: automorphisms and persistence}, Math.
  Ann. \textbf{381} (2021), no.~1-2, 439--457. \MR{4322617}

\bibitem[Jav21b]{Javanpeykar:FiniteMapsSurfaces}
\bysame, \emph{Finiteness of non-constant maps over a number field}, Preprint,
  arXiv:2112.11408 (to appear in \textit{Math.~Res.~Lett.}) (December 21,
  2021).

\bibitem[JK20]{JKam}
Ariyan Javanpeykar and Ljudmila Kamenova, \emph{Demailly's notion of algebraic
  hyperbolicity: geometricity, boundedness, moduli of maps}, Math. Z.
  \textbf{296} (2020), no.~3-4, 1645--1672. \MR{4159843}

\bibitem[JV21]{JVez}
Ariyan Javanpeykar and Alberto Vezzani, \emph{Non-archimedean hyperbolicity and
  applications}, J. Reine Angew. Math. \textbf{778} (2021), 1--29. \MR{4308620}

\bibitem[JX22]{JavanpeykarXie:Finiteness}
Ariyan Javanpeykar and Junyi Xie, \emph{Finiteness properties of
  pseudo-hyperbolic varieties}, Int. Math. Res. Not. IMRN (2022), no.~3,
  1601--1643. \MR{4373220}

\bibitem[Kaw80]{Kawamata}
Yujiro Kawamata, \emph{On {B}loch's conjecture}, Invent. Math. \textbf{57}
  (1980), no.~1, 97--100. \MR{564186}

\bibitem[Kaw81]{Kawamata:Characterization}
\bysame, \emph{Characterization of abelian varieties}, Compositio Math.
  \textbf{43} (1981), no.~2, 253--276. \MR{622451}

\bibitem[KO75]{KobayashiOchiai:MeromorphicMappings}
Shoshichi Kobayashi and Takushiro Ochiai, \emph{Meromorphic mappings onto
  compact complex spaces of general type}, Invent. Math. \textbf{31} (1975),
  no.~1, 7--16. \MR{402127}

\bibitem[Kob67]{Kobayashi:Intrinsic}
Shoshichi Kobayashi, \emph{Intrinsic metrics on complex manifolds}, Bull. Amer.
  Math. Soc. \textbf{73} (1967), 347--349. \MR{210152}

\bibitem[Kob98]{Kobayashi}
\bysame, \emph{Hyperbolic complex spaces}, Grundlehren der Mathematischen
  Wissenschaften [Fundamental Principles of Mathematical Sciences], vol. 318,
  Springer-Verlag, Berlin, 1998. \MR{1635983}

\bibitem[KRZB16]{KRZB:Uniform}
Eric Katz, Joseph Rabinoff, and David Zureick-Brown, \emph{Uniform bounds for
  the number of rational points on curves of small {M}ordell-{W}eil rank}, Duke
  Math. J. \textbf{165} (2016), no.~16, 3189--3240. \MR{3566201}

\bibitem[K{\"u}h21]{Kuhne:Equidistribution}
Lars K{\"u}hne, \emph{Equidistribution in families of abelian varieties and
  uniformity}, Preprint, arXiv:2101.10272v3 (September 5, 2021).

\bibitem[Lan86]{Lang:HyperbolicDiophantineAnalysis}
Serge Lang, \emph{Hyperbolic and {D}iophantine analysis}, Bull. Amer. Math.
  Soc. (N.S.) \textbf{14} (1986), no.~2, 159--205. \MR{828820}

\bibitem[Laz04]{Lazarsfeld:Positivity2}
Robert Lazarsfeld, \emph{Positivity in algebraic geometry. {II}}, Ergebnisse
  der Mathematik und ihrer Grenzgebiete. 3. Folge. A Series of Modern Surveys
  in Mathematics [Results in Mathematics and Related Areas. 3rd Series. A
  Series of Modern Surveys in Mathematics], vol.~49, Springer-Verlag, Berlin,
  2004, Positivity for vector bundles, and multiplier ideals. \MR{2095472}

\bibitem[LV20]{LawrenceVenkatesh:Mordell}
Brian Lawrence and Akshay Venkatesh, \emph{Diophantine problems and {$p$}-adic
  period mappings}, Invent. Math. \textbf{221} (2020), no.~3, 893--999.
  \MR{4132959}

\bibitem[Mae83]{Maehara:FinitenessProperty}
Kazuhisa Maehara, \emph{A finiteness property of varieties of general type},
  Math. Ann. \textbf{262} (1983), no.~1, 101--123. \MR{690010}

\bibitem[Moi49]{Moise:Grille}
Edwin~E. Moise, \emph{Grille decomposition and convexification theorems for
  compact metric locally connected continua}, Bull. Amer. Math. Soc.
  \textbf{55} (1949), 1111--1121. \MR{35430}

\bibitem[Mor22]{Mordell1992rational}
Louis Mordell, \emph{On the rational solutions of the indeterminate equation of
  the third and fourth degree}, Proc. Camb. Phil. Soc., vol.~21, 1922,
  pp.~179--192.

\bibitem[Mor21]{MorrowNonArchGGLV}
Jackson~S. Morrow, \emph{Non-archimedean entire curves in closed subvarieties
  of semi-abelian varieties}, Math. Ann. \textbf{379} (2021), no.~3-4,
  1003--1010. \MR{4238258}

\bibitem[MR23]{MorrowRosso:Special}
Jackson~S. Morrow and Giovanni Rosso, \emph{A non-{A}rchimedean analogue of
  {C}ampana's notion of specialness}, Algebraic Geometry \textbf{10} (2023),
  no.~3, 262--297.

\bibitem[Nog98]{Nogu}
Junjiro Noguchi, \emph{On holomorphic curves in semi-abelian varieties}, Math.
  Z. \textbf{228} (1998), no.~4, 713--721. \MR{1644444}

\bibitem[Och77]{Ochiai77}
Takushiro Ochiai, \emph{On holomorphic curves in algebraic varieties with ample
  irregularity}, Invent. Math. \textbf{43} (1977), no.~1, 83--96. \MR{0473237}

\bibitem[Pac04]{Pacienz:SubvarietiesGeneralType}
Gianluca Pacienza, \emph{Subvarieties of general type on a general projective
  hypersurface}, Trans. Amer. Math. Soc. \textbf{356} (2004), no.~7,
  2649--2661. \MR{2052191}

\bibitem[Pet09]{petsche2009nonarchimedean}
Clayton Petsche, \emph{Non-{A}rchimedean equidistribution on elliptic curves
  with global applications}, Pacific journal of mathematics \textbf{242}
  (2009), no.~2, 345--375.

\bibitem[Poi13]{Poineau:AngelicBerkovich}
J\'{e}r\^{o}me Poineau, \emph{Les espaces de {B}erkovich sont ang\'{e}liques},
  Bull. Soc. Math. France \textbf{141} (2013), no.~2, 267--297. \MR{3081557}

\bibitem[RTW21]{RousseauTurchetWang:Nonspecial}
Erwan Rousseau, Amos Turchet, and Julie Tzu-Yueh Wang, \emph{Nonspecial
  varieties and generalised {L}ang-{V}ojta conjectures}, Forum Math. Sigma
  \textbf{9} (2021), Paper No. e11, 29. \MR{4215673}

\bibitem[RV20]{Vazquez_HyperbolicityNotions}
Rita Rodr\'{\i}guez~V\'{a}zquez, \emph{Hyperbolicity notions for varieties
  defined over a non-{A}rchimedean field}, Michigan Math. J. \textbf{69}
  (2020), no.~1, 41--78. \MR{4071345}

\bibitem[RV21]{Vazquez_NormalFamily}
\bysame, \emph{Non-{A}rchimedean normal families}, Ann. Inst. Fourier
  (Grenoble) \textbf{71} (2021), no.~4, 1677--1732. \MR{4398245}

\bibitem[Sch92]{Schneider:Symmetric}
Michael Schneider, \emph{Symmetric differential forms as embedding obstructions
  and vanishing theorems}, J. Algebraic Geom. \textbf{1} (1992), no.~2,
  175--181. \MR{1144433}

\bibitem[{Sta}15]{stacks-project}
The {Stacks Project Authors}, \emph{\emph{{S}tacks {P}roject}},
  http://stacks.math.columbia.edu, 2015.

\bibitem[Tak06]{Takayama:Pluriccanonical}
Shigeharu Takayama, \emph{Pluricanonical systems on algebraic varieties of
  general type}, Invent. Math. \textbf{165} (2006), no.~3, 551--587.
  \MR{2242627}

\bibitem[Tat71]{TateRigid}
John Tate, \emph{Rigid analytic spaces}, Invent. Math. \textbf{12} (1971),
  257--289. \MR{0306196}

\bibitem[Tem15]{Temkin:IntroBerk}
Michael Temkin, \emph{Introduction to {B}erkovich analytic spaces}, Berkovich
  spaces and applications, Lecture Notes in Math., vol. 2119, Springer, Cham,
  2015, pp.~3--66. \MR{3330762}

\bibitem[Tsu06]{Tsuji:Pluricanonical1}
Hajime Tsuji, \emph{Pluricanonical systems of projective varieties of general
  type. {I}}, Osaka J. Math. \textbf{43} (2006), no.~4, 967--995. \MR{2303558}

\bibitem[Tsu07]{Tsuji:Pluricanonical2}
\bysame, \emph{Pluricanonical systems of projective varieties of general type.
  {II}}, Osaka J. Math. \textbf{44} (2007), no.~3, 723--764. \MR{2360948}

\bibitem[Uen73]{Ueno}
Kenji Ueno, \emph{Classification of algebraic varieties. {I}}, Compositio Math.
  \textbf{27} (1973), 277--342. \MR{0360582}

\bibitem[Voi96]{Voisin:ConjectureClemens2}
Claire Voisin, \emph{On a conjecture of {C}lemens on rational curves on
  hypersurfaces}, J. Differential Geom. \textbf{44} (1996), no.~1, 200--213.
  \MR{1420353}

\bibitem[Voj91]{Vojta:Mordell}
Paul Vojta, \emph{Siegel's theorem in the compact case}, Ann. of Math. (2)
  \textbf{133} (1991), no.~3, 509--548. \MR{1109352}

\bibitem[Voj15]{VojtaLangExc}
\bysame, \emph{A {L}ang exceptional set for integral points}, Geometry and
  analysis on manifolds, Progr. Math., vol. 308, Birkh\"{a}user/Springer, Cham,
  2015, pp.~177--207. \MR{3331400}

\bibitem[Xu94]{Xu:SubvarietiesGeneralHypersurfaces}
Geng Xu, \emph{Subvarieties of general hypersurfaces in projective space}, J.
  Differential Geom. \textbf{39} (1994), no.~1, 139--172. \MR{1258918}

\bibitem[Yeo22]{Yeong:AlgHyp5}
Wern Yeong, \emph{Algebraic hyperbolicity of very general hypersurfaces in
  products of projective spaces}, Preprint, arXiv:2203.01392 (March 2, 2022).

\bibitem[Zha06]{Zhang:DistributionsAlgebraicDynamics}
Shou-Wu Zhang, \emph{Distributions in algebraic dynamics}, Surveys in
  differential geometry. {V}ol. {X}, Surv. Differ. Geom., vol.~10, Int. Press,
  Somerville, MA, 2006, pp.~381--430. \MR{2408228}

\end{thebibliography}
\bibliographystyle{amsalpha}

 \end{document}